\documentclass[12pt,a4paper,draft]{amsart}
\usepackage{amsmath}
\usepackage{amsthm}
\usepackage{mathtools} 
\usepackage{hyperref}
\hypersetup{
 pdfauthor={Csaba Schneider and Igor Martins Silva},
 pdftitle={Invariants of Nilpotent Lie Algebras via Geometry and Algebra with a Focus on Computation},
}

\usepackage{thm-restate}
\usepackage{tikz, pgfplots}
\usetikzlibrary{matrix,arrows,patterns,decorations.pathreplacing,3d}
\usepackage{longtable}
\usepackage{array, multirow}
\newcolumntype{L}[1]{>{\raggedright\let\newline\\\arraybackslash\hspace{0pt}}m{#1}}
\newcolumntype{C}[1]{>{\centering\let\newline\\\arraybackslash\hspace{0pt}}m{#1}}
\newcolumntype{R}[1]{>{\raggedleft\let\newline\\\arraybackslash\hspace{0pt}}m{#1}}

\usepackage{LobsterTwo}
\usepackage{bm} %boldsymbol substituto 

\usepackage{tikz-cd}
\usepackage{listings}
\usepackage[utf8]{inputenc}
\usepackage{textgreek}

\DeclareMathAlphabet{\mathcaln}{OMS}{zplm}{m}{n} %definindo mathcal normal

\usepackage[bitstream-charter]{mathdesign}
\usepackage[shortlabels]{enumitem}
\usepackage{xcolor}
\usepackage[makeroom]{cancel}

\usepackage{algorithm}

\pgfplotsset{compat=1.18}

\lstdefinelanguage{Sage}[]{Python}
{morekeywords={False,sage,True},sensitive=true}
\definecolor{dblackcolor}{rgb}{0.0,0.0,0.0}
\definecolor{dbluecolor}{rgb}{0.01,0.02,0.7}
\definecolor{dgreencolor}{rgb}{0.2,0.4,0.0}
\definecolor{dgraycolor}{rgb}{0.30,0.3,0.30}

\newcommand{\fraku}{\mathfrak{u}}

\newcommand{\bbC}{\mathbb{C}}

\newcommand{\bbF}{\mathbb{F}}

\newcommand{\bbQ}{\mathbb{Q}}

\newcommand{\bbZ}{\mathbb{Z}}
\newcommand{\calnV}{\mathcaln{V}}

\newcommand{\bfc}{\bm{c}}

\newcommand{{\bff}}{\mathbf{f}}

\newcommand{\bfp}{{\bm{p}}}
\newcommand{\bfq}{{\bm{q}}}

\newcommand{\bfx}{{\bm{x}}}
\newcommand{\bfy}{{\bm{y}}}

\newtheorem{theorem}{Theorem}[section]
\newtheorem{lemma}[theorem]{Lemma}
\newtheorem{corollary}[theorem]{Corollary}
\newtheorem{proposition}[theorem]{Proposition}

\theoremstyle{definition}

\newtheorem{example}[theorem]{Example}

\numberwithin{equation}{section}

\DeclareMathOperator{\ad}{ad}

\DeclareMathOperator{\Der}{Der}

\DeclareMathOperator{\Frac}{Frac}
\DeclareMathOperator{\lt}{lt}
\DeclareMathOperator{\lm}{lm}

\DeclareMathOperator{\Mat}{Mat}
\DeclareMathOperator{\trdeg}{trdeg}

\makeatletter
\newcommand{\bigcdot}{\mathpalette\bigcdot@{.5}}
\newcommand*\bigcdot@[2]{\mathbin{\vcenter{\hbox{\scalebox{#2}{$\m@th#1\bullet$}}}}}
\makeatother

\newcommand{\polx}{\bbF[\bm{x}_{n}]}

\newcommand{\frpolx}{\bbF(\bm{x}_{n})}

\newcommand{\derfrpolx}{\text{Der}_{\bbF}\big(\frpolx\big)}

\newcommand{\Ann}{\textrm{Ann}}

\title[Invariants of nilpotent Lie algebras]{Invariants of Nilpotent Lie Algebras via Geometry and Algebra with a Focus on Computation}
\author{Csaba Schneider and Igor Martins Silva}
\address{Departamento de Matem\'atica\\
Universidade Federal de Minas Gerais\\
31270-901 Belo Horizonte, MG, Brazil}
\email{csaba.schneider@gmail.com,\ silvamigor@gmail.com} 
\urladdr{schcs.github.io/WP/}
\date{7 September 2026}

\begin{document}

\begin{abstract}
We consider the problem of computing rational invariants of nilpotent Lie algebras. 
We compare two methods that are commonly used for this task: the method of integral curves and 
the Dixmier map. Given a derivation of a rational function field with polynomial coefficients, we formulate a condition under which the kernel can be recovered from a family of rational integral curves, and we show that triangular derivations satisfy this hypothesis.
This yields an explicit description of the kernel as a purely transcendental extension and produces algebraically independent generators. We also show that, in the triangular case, the resulting generators agree with those obtained from the Dixmier map via a local slice. 
A careful analysis of the generating set obtained from this method leads to 
an algorithm for computing generators of the rational invariant field of a nilpotent Lie algebra. 
An implementation of the methods is available in the SageMath system.
\end{abstract}
\maketitle

\section{Introduction}
\label{introduction}
Let \(L\) be a finite-dimensional Lie algebra over a field \(\bbF\) of characteristic zero. The adjoint action of \(L\) on itself extends, by the Leibniz rule, to an action of \(L\) by derivations on the symmetric algebra \(\bbF[L]\), and consequently on its field of fractions \(\bbF(L)\). The subalgebra \(\bbF[L]^L\) of polynomial invariants and the subfield \(\bbF(L)^L\) of rational invariants, consisting of the elements annihilated by \(\ad(x)\) for all \(x\in L\), record fundamental structural information about \(L\) and appear throughout mathematics and mathematical physics. On the mathematical side, Duflo's theorem identifies \(\bbF[L]^L\) with the center of the universal enveloping algebra \(U(L)\) for finite-dimensional Lie algebras over 
fields of characteristic zero, with direct consequences for the representation theory of \(L\); invariants are also used as practical tools for distinguishing non-isomorphic Lie algebras in classification programmes. On the physical side, polynomial invariants are precisely the Casimir operators of \(L\), whose eigenvalues classify the irreducible representations relevant to a given physical system and correspond to conserved quantities, such as angular momentum, associated with the symmetries encoded by \(L\).

Determining a minimal generating set for \(\bbF[L]^L\) explicitly is, in general, a difficult problem, and the algebra \(\bbF[L]^L\) may not even be finitely generated: examples of this phenomenon, related to Hilbert's fourteenth problem, arise as kernels of locally nilpotent derivations~\cite{Rob90,ACN25}. The rational invariant field \(\bbF(L)^L\) is more tractable: a classical theorem of Dixmier~\cite{Dix57} shows, 
for nilpotent Lie algebras $L$, that $\bbF(L)^L$  is always a purely transcendental extension of \(\bbF\), and \(\bbF(L)^L\) still carries useful information about \(\bbF[L]^L\), to which it is closely related. We approach this problem with the explicit aim of 
developing an algorithm to compute generators for the rational invariant field of a nilpotent Lie algebra
that is suitable for machine computation.

In the theoretical physics literature, the invariants of a Lie algebra are usually approached through the \emph{method of integral curves}, also known as the \emph{method of characteristics}: since an invariant is, by definition, a first integral of the system of linear partial differential equations \(\ad(x)(f)=0\), \(x\in L\), one integrates the corresponding vector fields and looks for functions that are constant along the resulting curves. This approach is described at length, and applied to numerous families of low-dimensional Lie algebras, in the monograph of \v{S}nobl and Winternitz~\cite[Chapter~3]{SW14}. In that treatment, and in much of the related physics literature, the method is applied on a case-by-case basis: for each Lie algebra, the associated system of ordinary differential equations is integrated by hand, candidate invariants are guessed from the resulting integral curves and then verified directly. No general algorithm, however, is given, nor are precise conditions formulated under which the method is guaranteed to succeed and to produce a complete set of algebraically independent rational invariants. One of the main purposes of the present paper is to fill this gap: we isolate an explicit condition, which we call ``Condition IC'', under which the method of integral curves can be turned into a rigorous and constructive procedure for computing \(\bbF(L)^L\), and we show that this condition is satisfied whenever \(L\) is nilpotent.

In the algebra literature, by contrast, the kernel of a derivation is more commonly computed by means of the \emph{Dixmier map}, a Taylor-series type construction associated with a slice of a locally nilpotent derivation~\cite[Section~1.1.9]{Fre17}. We show that, for triangular derivations, these two traditions are in fact two sides of the same coin: the generic integral curve of a triangular derivation, evaluated at the time at which it meets a fixed hyperplane transverse to the flow, produces exactly the same generators as the Dixmier map associated with the corresponding slice (Theorem~\ref{th:dix}). This reconciles the analytic, curve-theoretic point of view favoured in the physics literature with the algebraic, derivation-theoretic point of view favoured in the algebra literature, and it shows that neither method is more general than the other in this setting.

Building on this correspondence, we develop an algorithm to compute algebraically independent generators of \(\bbF(L)^L\) for an arbitrary finite-dimensional nilpotent Lie algebra \(L\). The key observation is that a nilpotent Lie algebra admits a triangular basis \(x_1,\ldots,x_n\), adapted to the upper central series of \(L\), with respect to which each of the derivations \(\ad(x_1),\ldots,\ad(x_n)\) becomes triangular once restricted to the field of invariants of its predecessors. Applying the method of integral curves successively along this sequence of triangular derivations yields explicit algebraically independent generators of \(\bbF(L)^L\), exhibits it as a purely transcendental extension of \(\bbF\), and recovers the Beltrametti--Blasi formula~\cite{BB66} for the transcendence degree of \(\bbF(L)^L\) in terms of the rank of the multiplication matrix of \(L\). We have implemented the resulting algorithm in the computational algebra system SageMath and used it to compute invariants of nilpotent Lie algebras of dimension up to 6, and of the algebra of strictly upper triangular matrices in higher dimensions.

The paper is organised as follows. In Section~\ref{metodoCaract} we introduce the algebraic version of the method of integral curves for a derivation \(D\) of \(\frpolx\) with polynomial coefficients, formulate Condition~IC, and show that it allows one to identify \(\frpolx^D\) with the field generated by the pullback of coordinate functions along a suitable dominant rational map. In Section~\ref{sec:triang} we introduce triangular derivations, show that they satisfy Condition~IC by means of a generic integral curve, and prove that the resulting generators coincide with those produced by the Dixmier map (Theorem~\ref{th:dix}). Section~\ref{derivTriang_invPol} explores the connection with the algebra of polynomial invariants. In Section~\ref{sec:inv_nilp_lie} we specialise to nilpotent Lie algebras and prove our main structural result (Theorem~\ref{teor_presDerTriangAd}), which shows that a triangular basis produces a sequence of triangular derivations whose successive kernels compute \(\bbF(L)^L\). In Section~\ref{gerAlgInvarRac_consAlg} we deduce the Beltrametti--Blasi formula for the transcendence degree of \(\bbF(L)^L\). Finally, Section~\ref{sec:algorithm} describes the resulting algorithm, its implementation in SageMath, and several sample computations.

\subsection*{Acknowledgement}
The first 
author acknowledges the support of a {\em Bolsa de Produtividade em Pesquisa} (CNPq, Brazil, 309251/2023-0) and a {\em Projeto Universal (FAPEMIG, Brazil, APQ-03491-25)}. 
The second author was supported by a PhD scholarship of CAPES, Brazil. We are grateful 
to Artem Lopatin, Arturo Ulises Fernandez Perez, Elizaveta Vishnyakova, José Luis Vilca Rodríguez,  Renato Vidal da Silva Martins for their thoughtful comments. 
We are also grateful to 
Travis Scrimshaw for his help with the SageMath implementation.

\section{Kernels of derivations and integral curves}
\label{metodoCaract}
%%%%%%%%%%%%%%%%%%%%%%%%%%%%%%%%%
%
%       MÉTODO DAS CARACTERÍSTICAS - HIPÓTESES DE APLICAÇÃO
%
%%%%%%%%%%%%%%%%%%%%%%%%%%%%%%%%%

Suppose that $\bbF$ is a field of characteristic zero, which need not be algebraically closed.
Consider the vector space $\bbF^n$ with a fixed basis, and let $x_1,\ldots,x_n$ denote the corresponding coordinate functions
$\bbF^n\to \bbF$. In this way, $\polx$ denotes the algebra of polynomial functions and $\frpolx$ the field of rational functions on $\bbF^n$. For $i\in\{1,\ldots,n\}$, let $D_{x_i}$ denote the 
partial derivative with respect to $x_i$, and let
\[ 
D = \sum_{i=1}^n f_i D_{x_i}
\] 
be a derivation of $\frpolx$, where $f_i \in \polx$ for all $i$. Since we are interested in $\ker D$, we assume that the coefficients of $D$ are polynomial functions. However, one could more generally assume that the coefficients are rational functions. 

Our goal in this section is to describe the kernel of $D$. To achieve this,
we formulate an algebraic version of the method of integral curves, 
which is well known in the analytic theory of partial differential equations. In the analytic setting, the method of 
integral curves covers an open region $U$ with integral curves of the vector field $\bm{F} = (f_1, \ldots, f_n)$ 
and seeks functions $f$ that are constant along each of these integral curves. 
Such functions $f$ are solutions of the equation $D(f) = 0$ (see \cite[Section 3.2]{Eva10}). 

In our algebraic context, this method must be applied with sufficient care, as the integral curves are often not 
algebraic curves, and thus the resulting functions are not necessarily rational. 
To address this, we present an algebraic interpretation of the method that works under certain geometric conditions. 
These conditions will be shown to hold for triangular derivations, as demonstrated in Section~\ref{sec:triang}.

\subsection{Rational and birational maps} 
\label{metodoCaract_ratBirMaps} 
Topological notions for subsets of $\bbF^n$ are understood in the Zariski topology. 
For a polynomial $f$, we let $\calnV(f)$ denote the Zariski closed zero set of $f$.
Suppose that $X \subseteq \bbF^n$ and $Y \subseteq \bbF^m$ are closed, irreducible subsets. 
A {\em rational map} $f: X \dashrightarrow Y$ is a morphism defined on a Zariski open subset $U \subseteq X$ with values in $Y$. 
Rational maps are denoted by dashed arrows. The components of such a rational map are functions of the form $p/q$, 
where $p, q \in \polx$ are polynomials and $q$ does not vanish anywhere on $U$. 

A rational map $f: X \dashrightarrow Y$ is said to be {\em dominant} if $\mbox{Im}(f)$ contains a dense open subset of $Y$. (This is deliberately stronger than assuming that $\mbox{Im}(f)$ is dense in $Y$.)
Let $f: X \dashrightarrow Y$ and $g: Y \dashrightarrow Z$ be rational maps. If 
$\mbox{dom}(f) \cap f^{-1}(\mbox{dom}(g)) \neq \emptyset$ (which happens, for example, when $f$ is dominant), 
then this intersection is a non-empty Zariski open subset of $X$, and the composition $g \circ f: X \dashrightarrow Z$ is defined 
and is itself a rational map. In this case, we say that $f$ and $g$ are {\em composable}. 

A map $f: X \dashrightarrow Y$ is said to be {\em birational} if there exists a rational inverse $g: Y \dashrightarrow X$ such that 
there are Zariski open subsets $U \subseteq X$ and $V \subseteq Y$ satisfying $f \circ g = \mbox{id}_V$ and $g \circ f = \mbox{id}_U$. 
In this case, we say that $X$ and $Y$ are {\em birationally equivalent}. Note that in this case 
$\mbox{Im}(f)$ contains the dense open set $V$, and one may take $V=\mbox{dom}(g)$.

A {\em rational curve} $C$ is a closed, irreducible subset of $\bbF^n$ that is birationally equivalent to $\bbF$. 
A birational equivalence $\bfc: \bbF \dashrightarrow C$ is called a rational parametrization of $C$.

\subsection{Lie derivatives of rational functions}\label{metodoCaract_lieDerRat} 
Suppose that $f_1,\ldots,f_n\in\polx$. The functions $f_i$ determine a vector field $\bm F$ on $\bbF^n$ defined by 
\[ 
{\bm F}({\bm p})=(f_1({\bm p}),\ldots,f_n({\bm p}))
\] 
for all ${\bm p}\in \bbF^n$. With the vector field $\bm F$, there is an associated derivation $D_{{\bm F}}\in\Der_{\bbF}\big(\frpolx\big)$; namely, 
\[ 
D_{{\bm F}} = \sum_{i=1}^{n}f_{i}D_{x_{i}}.
\]   
Suppose that $C\subseteq \bbF^n$ is a rational curve. We say that 
%$C$ is an {\em integral curve} of $\bm F$ if, 
%at every point $\bm p\in C$, the tangent line $T_{\bm p}C$ of $C$ is parallel to the vector $\bm F(\bm p)$. 
%More concretely, 
$C$ is an integral curve for the vector field $\bm F$ if whenever 
$\bfc:\bbF\dashrightarrow C$ is a birational
equivalence (rational parametrization),
%such that 
%\[ 
%\bfc(t)=(c_1(t),\ldots,c_n(t))
%\] 
%with $c_i\in\bbF(t)$, 
then, letting \(D_{t}\) be the derivative with respect to \(t\), 
there is a rational function $\lambda_{\bfc}:\mbox{dom}(\bfc)\to \bbF$ such that   
\begin{equation}\label{eq:int_curve}
        D_{t}(\bfc) = \lambda_{\bfc}{\bm F}\circ \bfc.
\end{equation}
That is, the tangent vector of the parametrization $\bfc$ at a point $\bm p\in C$
is a scalar multiple of the vector of the vector field ${\bm F}$ at $\bm p$.  

The Lie derivative $L_{\bm F}$ acts on rational functions $f\in\frpolx$ as 
\[ 
L_{\bm F}\cdot f=D_{{\bm F}}(f).
\] 
The kernel of the derivation $D_{{\bm F}}$ consists of those rational functions $f\in\frpolx$ such that 
\[ 
L_{\bm F}\cdot f=D_{\bm F}(f)=0.
\]  
As is common in invariant theory, the kernel of a derivation $D_{{\bm F}}$ will be denoted by $\frpolx^{D_{{\bm F}}}$. If we restrict a 
derivation $D_{{\bm F}}$ to polynomial functions, then the kernel of the restriction will be denoted by $\polx^{D_{\bm F}}$. 
One may verify by simple calculation that the functions in $\frpolx^{D_{\bm F}}$ form a field 
and $\polx^{D_{\bm F}}$ is an $\bbF$-algebra inside $\frpolx^{D_{\bm F}}$. 
In the theory of differential equations, a function $f$ that 
is annihilated by the vector field ${\bm F}$ is often called \textit{a first integral} of ${\bm F}$ (see for example~\cite[Section~2.10.5]{Arn06}). In this section we investigate the problem of 
obtaining a generating system for the field of rational first integrals of ${\bm F}$.

\begin{example} \label{exem_campVetFx}
    Consider the derivation
    \[
        D = x_{1}D_{x_{2}} + x_{2}D_{x_{3}} \in \Der_{\bbF}\big(\bbF(x_{1},x_{2},x_{3})\big).
    \]
    Note that $D$ is associated with a vector field $\bm F$ that relates the point \(\bm{p} = (p_{1}, p_{2}, p_{3}) \in \bbF^{3}\) to the vector \((0, p_{1}, p_{2}) \in T_{\bm{p}}\bbF^{3}\), as shown in the figure below.

    \vspace{5pt}
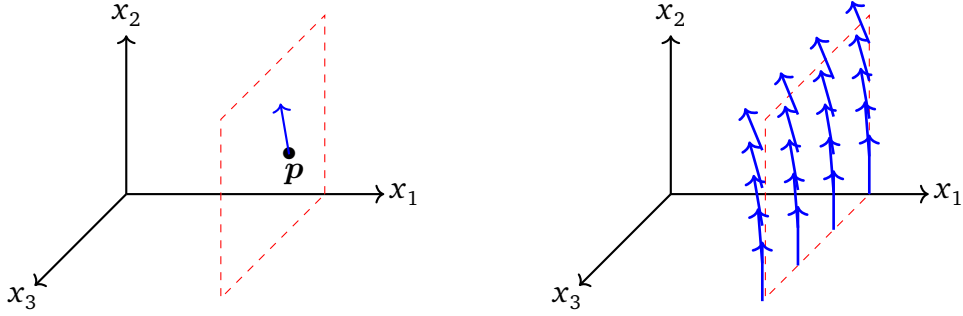
    \begin{center}
        \begin{minipage}{.45\textwidth}%
            \begin{flushleft}
                \begin{tikzpicture}[scale=0.35]
                    \draw node at (10.5,0,0) {\(x_{1}\)};
                    \draw node at (0,6.75,0) {\(x_{2}\)};
                    \draw node at (0,0,10.2) {\(x_{3}\)};
                    \draw [thick] (0,0,0) -- (3.6,0,0);
                    \draw [thick] (3.6,0,0) -- (7.5,0,0);
                    \draw [thick, ->] (7.5,0,0) -- (9.75,0,0);
                    \draw [thick, ->] (0,0,0) -- (0,6,0);
                    \draw [thick, ->] (0,0,0) -- (0,0,9);
                    \draw [red, dashed] (7.5,0,0) -- (7.5,6.75,0) -- (7.5,6.75,10.2) -- (7.5,0,10.2) -- (7.5,0,0);
                    \draw node at (7.5, 2.85, 3.5) {\(\bullet\)};
                    \draw node at (7.5, 1.9, 3) {\(\bm{p}\)};
                    \draw [thick, ->, blue] (7.5, 2.85, 3.5) -- (7.5, 5.1, 4.355);
                \end{tikzpicture}
            \end{flushleft}
        \end{minipage}%
        \begin{minipage}{.45\textwidth}%
            \begin{flushleft}
                \begin{tikzpicture}[scale=0.35]
                    \draw node at (10.5,0,0) {\(x_{1}\)};
                    \draw node at (0,6.75,0) {\(x_{2}\)};
                    \draw node at (0,0,10.2) {\(x_{3}\)};
                    \draw [thick] (0,0,0) -- (3.6,0,0);
                    \draw [thick] (3.6,0,0) -- (7.5,0,0);
                    \draw [thick, ->] (7.5,0,0) -- (9.75,0,0);
                    \draw [thick, ->] (0,0,0) -- (0,6,0);
                    \draw [thick, ->] (0,0,0) -- (0,0,9);
                    \draw [red, dashed] (7.5,0,0) -- (7.5,6.75,0) -- (7.5,6.75,10.2) -- (7.5,0,10.2) -- (7.5,0,0);
                    \pgfmathsetmacro{\s}{0.3}
                    \foreach \i in {7.5}
                        \foreach \j in {0, 1.425, 2.85, 4.275, 5.7}
                            \foreach \k in {0, 3.5, 7, 10.5}
                                \draw[->, color=blue, line width=1pt] 
                                    (\i, \j, \k) -- (\i+ \s*0, \j + \s*\i, \k  + \s*\j);
                \end{tikzpicture}
            \end{flushleft}
        \end{minipage}%
    \end{center}
    \vspace{5pt}

    \noindent
    To find parametrized integral curves of \(\bm F\), consider \(\bfc = (c_{1}, c_{2}, c_{3}) \in \bbF(t)^{3}\) that satisfies the following system
    of ordinary differential equations:
    \[
        \begin{array}{ccccc}
            D_{t}(c_{1}) \, = \, 0, &&
            D_{t}(c_{2}) \, = \, c_{1}, &&
            D_{t}(c_{3}) \, = \, c_{2}.
        \end{array}
    \]
    By integration, we obtain as a solution to this system the family of parametrized rational curves \((\bfc_{\bm{p}})_{\bm{p} \in \bbF^{3}}\), where
    \[
        \bfc_{\bm{p}} = (c_{\bm{p}1}, c_{\bm{p}2}, c_{\bm{p}3}) = \left(p_{1},\; p_{1}t + p_{2},\; p_{1}\dfrac{t^{2}}{2} + p_{2}t + p_{3}\right).
    \]
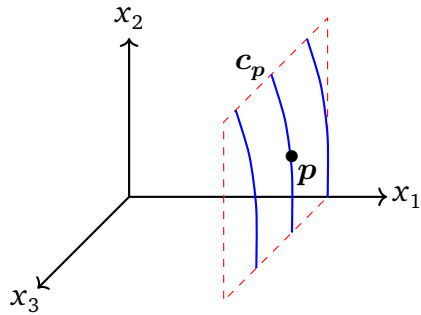
    \begin{center}
        \begin{tikzpicture}[scale=0.35]
            \draw node at (10.5,0,0) {\(x_{1}\)};
            \draw node at (0,6.75,0) {\(x_{2}\)};
            \draw node at (0,0,10.2) {\(x_{3}\)};
            \draw [thick] (0,0,0) -- (3.6,0,0);
            \draw [thick] (3.6,0,0) -- (7.5,0,0);
            \draw [thick, ->] (7.5,0,0) -- (9.75,0,0);
            \draw [thick, ->] (0,0,0) -- (0,6,0);
            \draw [thick, ->] (0,0,0) -- (0,0,9);
            \draw [red, dashed] (7.5,0,0) -- (7.5,6.75,0) -- (7.5,6.75,10.2) -- (7.5,0,10.2) -- (7.5,0,0);
            \draw [thick, blue] plot [smooth] coordinates {(7.5, 0, 3.5) (7.5, 2.25, 3.5) (7.5, 4.5, 4.175) (7.5, 6.75, 5.525)};
            \draw [thick, blue] plot [smooth] coordinates {(7.5, 0, 7) (7.5, 2.25, 7) (7.5, 4.5, 7.675) (7.5, 6.75, 9.025)};
            \draw [thick, blue] plot [smooth] coordinates {(7.5, 0, 0) (7.5, 2.25, 0) (7.5, 4.5, 0.675) (7.5, 6.75, 2.025)};
            \draw node at (7.5, 2.85, 3.5) {\(\bullet\)};
            \draw node at (7.8, 1.9, 2.8) {\(\bm{p}\)};
            \draw node at (6.9, 7.2, 6) {\(\bm{c}_{\bm{p}}\)};
        \end{tikzpicture}
    \end{center}
    Note that the resulting parametrized integral curves are full curves, that is, they are maps $\bbF\to\bbF^3$ and the family 
    of curves is parametrized by the triple $(p_1,p_2,p_3)\in\bbF^3$. 
    It is easy to verify that if 
    ${\bm q}\in \mbox{Im}(\bfc_{\bm p})$ for some ${\bm p},{\bm q}\in\bbF^3$, then $\bfc_{\bm p}(\bbF)=\bfc_{\bm q}(\bbF)$.  

    Let us parametrize the integral curves $\bfc$ so that $c_2(0)=0$; that is, $p_2=0$. This way, we have a family of integral curves 
    parametrized by $\gamma=(y_1,y_2)\in\bbF^2$ such that 
    \[ 
    \bfc_{\gamma}(t)=\left(y_1,y_1t,y_1\frac{t^2}2+y_2\right).
    \] 
    Thus, for each $\gamma=(y_1,y_2)\in\bbF^2$, we have that $C_\gamma=\mbox{Im}(\bfc_\gamma)$ is an integral curve for the vector field $\bm F$. 
    This also gives us a map   
    \[ 
     \psi:\bbF^3\to\bbF^3, \quad (y_1,y_2,t)\mapsto \left(y_1,y_1t,y_1\frac{t^2}2+y_2\right)
    \]
    with $\mbox{dom}(\psi)=\bbF^3$. Note that $\psi$ maps a triple $(y_1,y_2,t)$ to the point $\bfc_\gamma(t)$. 
    Its rational inverse is  
    \[ 
     \varphi:\bbF^3\dashrightarrow\bbF^3,\quad    (x_1,x_2,x_3)\mapsto (y_1,y_2,t)=\left(x_1,x_3-\frac{x_2^2}{2x_1},\frac{x_2}{x_1}\right)
    \] 
    with $U=\mbox{dom}(\varphi)=\bbF^3\setminus\mathcaln V(x_1)$. We have that $\psi\circ\varphi=\mbox{id}_U$ and $\varphi\circ\psi = \mbox{id}_{\bbF^3\setminus\mathcaln V(y_1)}$. Let 
    $\pi:\bbF^3\to \bbF^2$ denote the projection map to the first two coordinates and let $\chi=\pi\circ\varphi:U\to \bbF^2$ (also considered as 
    a rational map $\bbF^3\dashrightarrow\bbF^2$). For a given $\gamma=(y_1,y_2)\in\chi(U)$, 
    the fiber $\chi_\gamma=\chi^{-1}(\gamma)\subseteq U$ is the integral curve $C_\gamma$. Thus the open set $U$ is foliated by 
    the integral curves $C_\gamma$ of the vector field $\bm F$. 
    
    In this example, each integral curve $C_\gamma$ has a rational 
    parametrization $\bfc_\gamma:\bbF\dashrightarrow C_\gamma$ such that $D_t\bfc_\gamma=\bm F\circ \bfc_\gamma$; that is, 
    the tangent vector of $\bfc_\gamma$ at a point $\bm p\in U$ coincides with the vector $\bm F(\bm p)$ of the field. 
    That is, the function $\lambda_{\bfc}$ in 
    equation \eqref{eq:int_curve} is the constant function $\lambda_{\bfc}=1$.
\end{example}

\begin{example}\label{ex:expfunc}
    Let us now consider the case $\bbF=\bbC$ and the derivation 
    \[ 
    D=x_1D_{x_1}+x_2D_{x_2}.
    \] 
    To determine the integral curves of the associated vector field $\bm F=(x_1,x_2)$, we consider the system of ordinary differential equations 
    \begin{equation}\label{eq:system2}
    D_t(c_1)=c_1\quad\mbox{and}\quad D_t(c_2)=c_2.        
    \end{equation}
    Its solutions are of the form 
    \[ 
    \bfc=(c_1,c_2),\quad c_1(t)=y_1\exp (t)\quad \mbox{and}\quad c_2(t)=y_2\exp (t)
    \]
    with constants $y_1,y_2\in\bbC$. 
    We may choose $y_2=1$ and this way our integral curves satisfy $c_2(0)=1$. This way, we obtain a family of integral 
    curves 
    $\bfc_y(t)=(y\exp(t),\exp(t))$ for each $y\in \bbC$. The image $C_y$ of $\bfc_y$ is the line 
    \[ 
    C_y=\{(ys,s)\mid s\in\bbC\setminus\{0\}\}
    \]
    with the point $(0,0)$ removed. Hence its closure 
    \[ 
    \overline C_y=\{(ys,s)\mid s\in\bbC\}
    \] 
    is an integral curve for the vector field $\bm F$. In fact, we may take the parametrization $s\mapsto (ys,s)$ for  
    $s\neq 0$ and $\lambda(s)=1/s$ to satisfy equation~\eqref{eq:int_curve}. 
    Let $U=\bbC^2\setminus\mathcaln V(x_2)$ and consider the map 
    \[ 
    \varphi:U\to \bbC^2,\quad (x_1,x_2)\mapsto (x_1/x_2,x_2).
    \] 
    Then $\varphi$ can be considered as a birational map $\bbC^2\dashrightarrow \bbC^2$. Let $\pi:\bbC^2\to \bbC$ be the 
    projection to the first coordinate and set $\chi=\pi\circ\varphi:\bbC^2\dashrightarrow\bbC$. Then for a given $y\in \chi(U)=\bbC$, we have that the 
    fiber $\chi_y=\chi^{-1}(y)$ is equal to $C_y=\overline C_y\cap U$, where $\overline C_y$ is the 
    closed integral curve as above. Thus in this case also, the open set $U$ 
    is foliated by integral curves, despite the fact that the solutions of the equations \eqref{eq:system2} do not have rational solutions. 
\end{example}
A notable property of the integral curves of a derivation \(D \in \derfrpolx\) manifests itself when evaluating a function belonging to the kernel of \(D\) along one of these curves, as shown in Lemma \ref{prop_curvasCarValConst}.
\begin{lemma} \label{prop_curvasCarValConst}
    Use the notation and assumptions introduced in Section~\ref{metodoCaract_lieDerRat}.
        If \(f \in \frpolx^{D_{{\bm F}}}\) and \(C\subseteq \bbF^n\) is an integral curve of \(\bm F\) such that $\mbox{dom}(f)\cap C\neq \emptyset$, then 
        \[
                f\big|_{\mbox{dom}(f)\cap C}\in \bbF.
        \]
        That is, \(f\) is constant along the integral curves.
\end{lemma}
\begin{proof}
Considering that $D_{{\bm F}}=\sum_{i=1}^n f_iD_{x_i}$, if \(f \in \frpolx^{D_{{\bm F}}}\), then 
\[ 
0=D_{{\bm F}}(f)=\sum_{i=1}^n f_iD_{x_i}(f).
\] 
Suppose that $\bfc:\bbF\dashrightarrow C$ is a birational equivalence and $\lambda_{\bfc}:\mbox{dom}(\bfc)\to \bbF$ is a 
function that satisfies equation~\eqref{eq:int_curve}. Let $C_0$ be a dense open subset of $C$ contained in 
$\mbox{Im}(\bfc)\cap \mbox{dom}(f)$; it suffices to show that $f\big|_{C_0}$ is constant.  
Applying the chain rule for the multivariable derivative, we also obtain that 
        \[
                D_{t}\left(f\circ \bfc\right) = \sum_{i=1}^{n}(D_{x_{i}}(f)\circ \bfc)\cdot D_{t}(c_{i})=\lambda_{\bfc}\sum_{i=1}^{n}(D_{x_{i}}(f)\circ\bfc)\cdot(f_i\circ\bfc)=\lambda_{\bfc} D_{\bm F}(f)\circ \bfc =0.
        \]
        Thus, \(D_{t}\left(f\circ \bfc\right) = 0\) and we conclude that \(f\circ\bfc\) is constant on $C_0$. 
        That is, $f$ is a constant function on a dense 
        open subset of $C$ which means that $f$ is a constant function on $\mbox{dom}(f)\cap C$. 
        %\hfill\(\boxtimes\)
\end{proof}

\subsection{The hypotheses of applying the method of integral curves}
\label{metodoCaract_hipoAplic} 
As mentioned at the beginning of this section, our goal is to describe \(\frpolx^{D}\) for a 
derivation $D\in\derfrpolx$ with polynomial coefficients by constructing algebraically independent generators 
for $\frpolx^D$. Using the notation introduced in Section~\ref{metodoCaract_lieDerRat}, we now present a hypothesis that allows us to achieve this goal. %
\begin{quotation}{\bf Condition IC.}
    Suppose that there exists a non-empty Zariski open set $U\subseteq \bbF^n$ and a dominant rational map
    \[ 
    \chi:\bbF^n\dashrightarrow \bbF^{n-1}
    \]  
    such that 
    \begin{enumerate}
        \item $U= \mbox{dom}(\chi)$;
        \item for all $\gamma=(\gamma_1,\ldots,\gamma_{n-1})\in\chi(U)$,
    the fiber $C_\gamma=\chi^{-1}(\gamma)\cap U$ is a dense open subset of a rational integral 
    curve for the vector field associated with the derivation $D$.
    \end{enumerate}
    \end{quotation}

%\begin{enumerate}[label={(H\arabic*)}]
%    \item \label{IC_H1} $U\subseteq \bbF^n$ is a Zariski open set and for every \(\bm{p} \in U\), and for each $\bm p\in U$ there is a unique integral curve 
%    $\bfc_{\bm p}\subseteq U$ such that $\bm p\in\bfc_{\bm p}$ (uniqueness refers to the image of the curve);
    %
%    \item \label{IC_H2} there exists an irreducible affine variety \(\Gamma \subseteq \bbF^{n}\), such that \(\Gamma \cap U \neq \varnothing\) and, for every \(\bm{p} \in U\), the intersection \(\bfc_{\bm{p}} \cap \Gamma\) contains exactly one element, denoted by \(\bm{p}_{0}\);
    %
%    \item \label{IC_H3} each component of the map
%    \begin{equation} \label{eq_intersCurvaGamma}
%        \begin{array}{rcl}
%            \varphi: \; U & \to & \Gamma \\
%            \bm{p} & \mapsto & \bm{p}_{0}
%        \end{array}
%    \end{equation}
%    is a rational function of \(\bbF^{n}\).
%\end{enumerate}

Suppose that $\chi:\bbF^n\dashrightarrow \bbF^{n-1}$ is as in Condition IC. 
If $\bm p\in U$ and $\chi(\bm p)=\gamma$, then $\bm p$ lies on the fiber $C_\gamma=\chi^{-1}(\gamma)\cap U$, which is a dense open subset of a rational integral curve $C$.
That is, one can say that the Zariski open set $U$ is foliated by the integral curves $C_{\gamma}$ with 
$\gamma\in\chi(U)$ as illustrated in the following image.
    
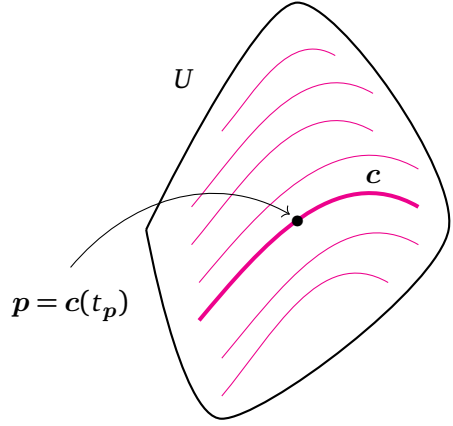
\begin{center}
        \begin{tikzpicture}
            \draw[thick] plot [smooth] coordinates {(-5,0) (-3,3) (-1,0) (-4,-2.5) (-5,0)};
            \draw node at (-4.5,2) {\small\(U\)};
            \draw[magenta] (-4,-2.2) to [out=50, in=150] (-1.8,-0.7);
            \draw[magenta] (-4,-1.7) to [out=50, in=150] (-1.4,-0.2);
            \draw[line width=1.5pt, magenta] (-4.3,-1.2) to [out=50, in=150] (-1.4,0.3);
            \draw[magenta] (-4.3,-0.7) to [out=50, in=150] (-1.4,0.8);
            \draw[magenta] (-4.4,-0.2) to [out=50, in=150] (-2,1.3);
            \draw[magenta] (-4.4,0.3) to [out=50, in=150] (-2,1.8);
            \draw[magenta] (-4,1.3) to [out=50, in=150] (-2.5,2.3);
            \draw node at (-3,0.1) {\small\(\bullet\)};
            \draw node at (-6,-1) {\small\(\bm{p} = \bfc(t_{\bm{p}})\)};
            \draw node at (-2,0.7) {\small\(\bfc\)};
            \draw[->] (-6,-0.5) to [out=50, in=150] (-3.1,0.2);
       \end{tikzpicture}
    \end{center}
    \begin{lemma}\label{lem:icurves}
        Assume that Condition IC above is valid and let 
        $f\in\frpolx$. 
        Then $D(f)=0$ if and only if $f\big|_{\mbox{dom}(f)\cap C_\gamma}$ is constant for each fiber $C_\gamma$ 
        such that $\mbox{dom}(f)\cap C_\gamma\neq \emptyset$.   
    \end{lemma}
    \begin{proof}
        If $D(f)=0$, then Lemma~\ref{prop_curvasCarValConst} shows that $f\big|_{\mbox{dom}(f)\cap C}$ is constant for 
        each integral curve $C$ of $D$ 
        such that $\mbox{dom}(f)\cap C\neq \emptyset$
        and hence 
        is constant on $C_\gamma$ for all $\gamma\in\chi(U)$ such that $\mbox{dom}(f)\cap C_\gamma\neq \emptyset$. 
        Assume, conversely, that $f\big|_{\mbox{dom}(f)\cap C_\gamma}$ is constant for each fiber $C_\gamma$ 
        such that $\mbox{dom}(f)\cap C_\gamma\neq \emptyset$. 
        Suppose that $C_\gamma$ is one of these fibers and is a dense open subset of the integral curve $C$. Let $\bfc:\bbF\dashrightarrow C$ be a rational parametrization and let 
        $\lambda_{\bfc}:\mbox{dom}(\bfc)\to \bbF$ be a rational function such that 
        \(D_{t}(c_{j}) = \lambda_{\bfc} f_{j}\circ \bfc\), for all \(j = 1, \dots, n\).
        Since the restriction of $f$ to the set $C_\gamma\cap \mbox{dom}(f)$ is constant, and this set is dense and open in 
        $C$, we obtain that $f\circ\bfc$ is constant on a dense open subset of $\mbox{dom}(\bfc)$, and hence \(D_{t}(f\circ \bfc) = 0\). We obtain, as in the proof of Lemma~\ref{prop_curvasCarValConst}, that
        \begin{equation} \label{eq_isoFGammaKerD_dtComp}
            0 = D_{t}\left(f\circ \bfc\right) = 
            \lambda_{\bfc} D(f)\circ \bfc.
            %\sum_{j=1}^{n}\left(D_{x_{j}}(f)\circ \bfc\right)\cdot D_{t}(c_{j}).
        \end{equation}
        %
%        Recalling that \(D = \sum_{i=1}^{n}f_{i}D_{x_{i}}\) and using that \(C\) is an integral curve of \(D\), we obtain that there is a rational function 
 %       $\lambda:\mbox{dom}(\bfc)\to \bbF$ such that 
%        \(D_{t}(c_{j}) = \lambda f_{j}\circ \bfc\), for all \(j = 1, \dots, n\). Thus, returning to Equation \eqref{eq_isoFGammaKerD_dtComp}, we have
%        %
%        \[
%            0=\sum_{j=1}^{n}\left(D_{x_{j}}(f)\circ \bfc\right)\cdot D_{t}(c_{j}) = \lambda\sum_{j=1}^{n}\left(f_{j}\circ \bfc\right) \left(D_{x_{j}}(f)\circ \bfc\right) = 
%            \lambda D(f)\circ \bfc;
%        \]
        %
        Since $\bfc$ is birational, it is non-constant, so $\lambda_{\bfc}$ is not identically zero. Therefore, on the non-empty open subset of $\mbox{dom}(\bfc)$ where $\lambda_{\bfc}\neq 0$, we have \(D(f)\circ \bfc = 0\). Hence $D(f)$ vanishes on a dense open subset of $C$, and therefore $D(f)|_C=0$. 
        By Condition IC and the argument following it, every point in \(U\) belongs to an integral curve, so  
        \(D(f)(\bm{p}) = 0\), for all \(\bm{p} \in U\cap \mbox{dom}(f)\). Again, since \(U\cap \mbox{dom}(f)\) is a non-empty Zariski open set of $\bbF^n$, 
        we obtain that \(D(f) = 0\).
    \end{proof}

%%%%%%%%%%%%%%%%%%%%%%%%%%%%%%%%%
%
%       METHOD OF CHARACTERISTICS - OBTAINING ALG. IND. GENERATORS
%
%%%%%%%%%%%%%%%%%%%%%%%%%%%%%%%%%

\subsection{Obtaining generators}
\label{metodoCaract_obtGerAlgInd}

%
%In this section, we will consider a derivation \(D \in \derfrpolx\) that satisfies the hypotheses of the previous section and we will construct an isomorphism between \(\frpolx^{D}\) and a finitely generated field, as presented in Proposition \ref{prop_isoFGammaKerD}. For this, we will make a brief observation.
%

%
%Let \(\bbF(\Gamma)\) be the field of rational functions of \(\Gamma\). Since \(\Gamma\) is an affine variety isomorphic to \(\bbF^{n-1}\), \(\bbF(\Gamma)\) can be seen as \(\bbF(y_{1}, \dots, y_{n-1})\). Let \(f_{0} \in \bbF(\Gamma)\), \(\den(f_{0})\) be the denominator of \(f_{0}\), and
%
%\[
%    \calnV\big(\den(f_{0})\big) = \{\bm{p} \in \Gamma \mid \den(f_{0})(\bm{p}) = 0\} \subseteq \Gamma
%\]
%
%the affine algebraic set associated with \(\den(f_{0})\). Also consider the map \(\varphi: U \to \Gamma\), as defined in Equation \eqref{eq_intersCurvaGamma}. Given \(\bm{p} \in U\setminus\varphi^{-1}\Big(\calnV\big(\den(f_{0})\big)\Big)\), we can consider \(f_{0}\big(\varphi(\bm{p})\big) \in \bbF\).
%
%
Suppose that Condition IC above is valid. 
Then $\chi:\bbF^n\dashrightarrow \bbF^{n-1}$ is a dominant rational map and  
induces an algebra homomorphism $\chi^*:\bbF(\bfx_{n-1})\to \frpolx$ between the fields of rational functions where $\chi^*$ is the pull-back. That is, if 
$f_0\in\bbF(\bfx_{n-1})$, then 
\[ 
\chi^*(f_0)=f_0\circ \chi\in\frpolx.
\]

%For an algebraic variety $\Gamma$, and a rational function $f_0\in\bbF(\Gamma)$ let us denote by $\den(f_{0})$ the denominator of 
%$f_0$ and by $\calnV\big(\den(f_{0})\big)$ the zero set of $\den(f_0)$. 

\begin{proposition} \label{prop_isoFGammaKerD}
    If Condition IC above is valid, then $\chi^*:\bbF(\bfx_{n-1})\to \frpolx^D$ is 
    an injective homomorphism of fields. Furthermore, if there exists a rational map $\psi:\bbF^{n-1}\dashrightarrow \bbF^n$ such that 
    $\chi\circ\psi$ is the identity map on a dense open subset $V\subseteq \bbF^{n-1}$ and 
    for all $f\in \frpolx^D$, we have that $\mbox{dom}(f)\cap \mbox{Im}(\psi)\neq \emptyset$, then $\chi^*$ is an isomorphism.
\end{proposition}
\begin{proof}
    The fact that $\chi^*$ is a homomorphism follows from the functorial correspondence between algebraic sets and their 
    fields of rational functions. Since $\bbF(\bfx_{n-1})$ is a field, the homomorphism $\chi^*$ is either zero or injective; 
    however, $\chi^*$ is clearly non-zero, and hence it must be injective.  
    Let $U$ be the open set of $\bbF^n$ as in Condition IC and suppose that $C_\gamma\subseteq U$ is the fiber for some 
    $\gamma\in\chi(U)$. 
    If $f_0\in\bbF(\bfx_{n-1})$, then 
    $\chi^*(f_0)(C_\gamma)=(f_0\circ \chi)(C_\gamma)=f_0(\gamma)$ and so $\chi^*(f_0)$ 
    restricted to $C_\gamma$ is constant. Since this is true for all integral curves 
    $C_\gamma$, Lemma~\ref{lem:icurves} shows that 
    $D(\chi^*(f_0))=0$, and hence $\chi^*(f_0)\in \frpolx^D$. 
    Suppose now that the additional condition of the proposition (the existence of the rational map $\psi$) 
    is also true and let us prove that $\chi^*:\bbF(\bfx_{n-1})\to \frpolx^D$ is surjective in this case.
    Let $f\in\frpolx$ such that $D(f)=0$. 
    By the additional conditions in the proposition, $f$ and $\psi$ are composable, and so taking $f_0=f\circ \psi$, $f_0$ is a rational function. Then, by Lemma~\ref{lem:icurves}, 
    $f$ is constant 
    on each integral curve $C_\gamma$. Consider the following diagram:
    \[
    \begin{tikzcd}
	\bbF^n  &&\bbF^{n-1} \\
	\bbF
    \arrow["\chi", dashed, bend left=10, from=1-1, to=1-3]
    \arrow["\psi", dashed, bend left=10, from=1-3, to=1-1]
	\arrow["{f_0}", dashed, from=1-3, to=2-1]
    \arrow["f"', dashed, from=1-1, to=2-1]
    \end{tikzcd}
\]
Since $\chi\circ\psi$ is the identity on the dense open set $V$, for $\bm p\in \mbox{dom}(f)\cap 
\chi^{-1}(V\cap\mbox{dom}(f_0))$, the point $\psi(\chi(\bm p))$ 
lies on the same fiber $C_{\chi(\bm p)}$ as $\bm p$; as $f$ is constant on each such fiber, we get $f(\psi(\chi(\bm p)))=f(\bm p)$. Hence the diagram 
above commutes; that is 
\[ 
f=f_0\circ\chi=\chi^*(f_0).
\] 
This shows that $\mbox{Im}(\chi^*)=\frpolx^D$.
\end{proof}

\begin{corollary} \label{coro_isoFGammaKerD}
    Assume that the hypotheses of Proposition \ref{prop_isoFGammaKerD} hold, including those of its ``Furthermore'' 
    part, so that $\chi^*:\bbF(\bfx_{n-1})\to\frpolx^D$ is an isomorphism. Then the elements    %
    \[
        \chi^*(x_{1})=x_1\circ\chi, \dots, \chi^*(x_{n-1})=x_{n-1}\circ\chi
    \]
    are algebraically independent generators of $\frpolx^D$. In particular, 
    $\frpolx^D$ is a purely transcendental extension of $\bbF$ of transcendence degree $n-1$. 
\end{corollary}

\begin{example} \label{exem_campVetFx2}
    Let us return to Examples~\ref{exem_campVetFx} and~\ref{ex:expfunc}. 
    In Example \ref{exem_campVetFx}, we considered the derivation 
    \[ 
    D = x_{1}D_{x_{2}} + x_{2}D_{x_{3}} \in \Der_{\bbF}\big(\bbF(\bm{x}_{3})\big).
    \]  
Now  $\chi:\bbF^3\dashrightarrow\bbF^2$ is the map 
\[ 
(x_1,x_2,x_3)\mapsto \left(x_1,x_3-\frac{x_2^2}{2x_1}\right). 
\] 
If $\psi:\bbF^2\to \bbF^3$, defined by $(y_1,y_2)\mapsto (y_1,0,y_2)$, then 
$\psi$ and $\chi$ satisfy the conditions of Proposition~\ref{prop_isoFGammaKerD}.
Thus, by Corollary~\ref{coro_isoFGammaKerD}, the elements
\[ 
\chi^*(x_1)=x_1\circ\chi=x_1\quad\mbox{and}\quad \chi^*(x_2)=x_2\circ\chi=x_3-\frac{x_2^2}{2x_1}
\] 
are algebraically independent generators of $\frpolx^D$.

In Example~\ref{ex:expfunc}, the derivation is 
\[ 
D=x_1D_{x_1}+x_2D_{x_2}.
\]
Here the function $\chi:\bbC^2\dashrightarrow\bbC$ is given by $(x_1,x_2)\mapsto x_1/x_2$
and $\psi:\bbF\to \bbF^2$ by $y\mapsto(y,1)$. Thus, by Corollary~\ref{coro_isoFGammaKerD}, $\frpolx^D$ is generated by 
\[ 
\chi^*(x_1)=x_1\circ\chi=x_1/x_2.
\] 
\end{example}

\section{Computing kernels of triangular derivations}\label{sec:triang}

%%%%%%%%%%%%%%%%%%%%%%%%%%%%%%%%%
%
%       TRIANGULAR DERIVATIONS
%
%%%%%%%%%%%%%%%%%%%%%%%%%%%%%%%%%

%
Triangular derivations of \(\frpolx\) form a special class of derivations that allow the application of the method of 
integral curves, presented in Section~\ref{metodoCaract}, 
enabling the construction of an algebraically independent set that generates their kernels. 
In our context, these derivations naturally arise when dealing with finite-dimensional nilpotent Lie algebras. 
For these algebras, it is possible to determine a basis that leads to a finite sequence of triangular derivations. 
The successive application of the method of integral curves to this sequence allows the construction of an algebraically 
independent set that generates the field of rational invariants. Therefore, in this section, we will explore this class of derivations in greater detail.
%

%%%%%%%%%%%%%%%%%%%%%%%%%%%%%%%%%
%
%       TRIANGULAR DERIVATIONS - CONTEXTUALIZATION AND DEFINITION
%
%%%%%%%%%%%%%%%%%%%%%%%%%%%%%%%%%

\subsection{Locally nilpotent and triangular derivations}
\label{derivTriang_contexDef}
Let $D$ be a derivation of an $\bbF$-algebra \(A\). We say that \(D\) is a {\em locally nilpotent derivation}, 
if for each \(a \in A\) there exists \(k \geq 1\) such that \(D^{k}(a) = 0\). 

%Locally nilpotent derivations play an important role in the response to  \textit{14th Hilbert's problem}, which raises the question:
%
%\begin{center}
%    if \(\bbK\) is a field such that \(\bbF \subseteq \bbK \subseteq \frpolx\), then is \(\bbK \cap \polx\) a finitely generated \(\bbF\)-algebra?
%\end{center}
%
%In 1990, Robert, in \cite{Rob90}, presented an example that negates this Hilbert's problem. This example, as shown by A'Campo-Neuen, 
%in \cite{ACN25}, is indeed the kernel of a locally nilpotent derivation.
%
%
%
%
%
Let \(\{a_{1}, \dots, a_{n}\} \subseteq A\) be an algebraically independent set and \(\bbF[\bm{a}_{n}]\) the subalgebra of \(A\) generated by 
\(\{a_{1}, \dots, a_{n}\}\). We say that a derivation \(D\) of the algebra \(\bbF[\bm{a}_{n}]\) is a {\em triangular derivation} with respect 
to the set \(\{a_{1}, \dots, a_{n}\}\), if \(D(a_{1}) \in \bbF\) and \(D(a_{i}) \in \bbF[a_{1}, \dots, a_{i-1}]\), for every \(i = 2, \dots, n\). 
A derivation \(D \in \Der_{\bbF}\big(\bbF[\bm{a}_{n}]\big)\) is a \textit{triangularizable derivation} if there exists an algebraically 
independent set \(\{a_{1}', \dots, a_{n}'\}\subseteq \bbF[\bm{a}_{n}]\) that also generates \(\bbF[\bm{a}_{n}]\) such that \(D\) is triangular with respect to the set \(\{a_{1}', \dots, a_{n}'\}\).
If \(A\) is a domain, we extend this definition to \(\Frac(A)\), 
saying that \(D \in \Der_{\bbF}\big(\Frac(A)\big)\) is a \textit{triangular derivation} with respect to the set \(\{a_{1}, \dots, a_{n}\}\), 
if \(D(a_{1}) \in \bbF\) and, for every \(i = 2, \dots, n\), we have that \(D(a_{i}) \in \bbF[a_{1}, \dots, a_{i-1}]\). 

%where \(D_{z_{i}} \in \derfrpolz\) is a derivation of the field \(\frpolx\) and \(D(z_{i})\) belongs to the polynomial algebra \(\polx\), satisfying \(D(z_{1}) \in \bbF\) and \(D(z_{i}) \in \bbF[z_{1}, \dots, z_{i-1}]\).
%

%
Proposition 3.29 of \cite{Fre17} shows that every triangularizable derivation is locally nilpotent, however, the 
converse is not true; see for example~\cite[Example~4.3]{df98}.
Knowing whether a derivation is triangularizable is difficult. However, there are works that seek this goal, as seen in the article \cite{Dai10}, where there is a characterization of such derivations for the case of derivations in \(\bbF[x_{1}, x_{2}, x_{3}]\). Additionally, still with a perspective on Hilbert's fourteenth problem, we have, in the article \cite{DF01}, a result that shows that the kernel of every triangular derivation in \(\bbF[x_{1}, x_{2}, x_{3}, x_{4}]\) is finitely generated as an \(\bbF\)-algebra.

In the subsequent discussion, triangular derivations of $\polx$ or of $\frpolx$ are meant to be 
with respect to the generating set $x_1,\ldots,x_n$. 
\begin{lemma} \label{lema_nucPolDerTriMonLidXk}
    Let $f_k,f_{k+1},\ldots,f_n\in\polx$ be polynomials such that $f_k\neq 0$ and let 
    \[
        D = f_{k}D_{x_{k}} + \cdots + f_{n}D_{x_{n}} \in \derfrpolx
    \]
    be a triangular derivation. Consider the lexicographic monomial order with respect to the order \(x_{1} < \cdots < x_{n}\). If \(f \in \polx^{D}\), then the leading monomial \(\lm(f)\) of $f$ does not contain the variable \(x_{k}\).
\end{lemma}
\begin{proof}
    Suppose that \(f \in \polx^{D}\) and let us assume, first, that \(f \in \bbF[x_{1}, \dots, x_{k}]\). Write \(f = x_{k}^{\ell}d_{\ell} + \cdots + x_{k}d_{1} + d_{0}\), where \(d_{i} \in \bbF[x_{1}, \dots, x_{k-1}]\), for every \(i = 0, \dots, \ell\). If \(\ell = 0\), then \(f = d_{0}\), so \(\lm(f)\) does not contain \(x_{k}\). Suppose that \(\ell \neq 0\). Note that \(D(d_{i}) = 0\), for every \(i = 0, \dots, \ell\), because \(d_{i} \in \bbF[x_{1}, \dots, x_{k-1}]\) and \(D(x_{i}) = 0\) for $i\in\{1,\ldots,k-1\}$. Thus,
    \begin{align} \label{eq_nucPolDerTriMonLidXk_1}
        0 &= D(f) = D(x_{k}^{\ell}d_{\ell} + x_{k}^{\ell-1}d_{\ell-1} + \cdots + x_{k}^{2}d_{2} + x_{k}d_{1} + d_{0}) = \notag \\[3pt]
        &= \ell x_{k}^{\ell-1}d_{\ell}D(x_{k}) + (\ell-1)x_{k}^{\ell-2}d_{\ell-1}D(x_{k}) + \cdots + 2x_{k}d_{2}D(x_{k}) + d_{1}D(x_{k}) = \notag \\[3pt]
        &= \ell x_{k}^{\ell-1}d_{\ell}f_{k} + (\ell - 1)x_{k}^{\ell-2}d_{\ell-1}f_{k} + \cdots + 2d_{2}x_{k}f_{k} + d_{1}f_{k}.
    \end{align}
    Analysing the coefficient of \(x_{k}^{\ell-1}\), we have that \(\ell d_{\ell}f_{k} = 0\), which is a contradiction, because \(\ell\), \(d_{\ell}\) and \(f_{k}\) are all non-zero. Therefore, \(\ell = 0\), that is, if \(f \in \polx^{D}\) and \(f \in \bbF[x_{1}, \dots, x_{k}]\), then \(\lm(f)\) does not contain \(x_{k}\).
    Let \(i \in \{k, \dots, n-1\}\) and assume, by induction, that if \(f \in \polx^{D}\) and \(f \in \bbF[x_{1}, \dots, x_{i}]\), then \(\lm(f)\) does not contain \(x_{k}\). Take \(f \in \polx^{D}\) such that \(f \in \bbF[x_{1}, \dots, x_{i+1}]\). Write \(f = x_{i+1}^{\ell}d_{\ell} + \cdots + x_{i+1}d_{1} + d_{0}\), where \(d_{j} \in \bbF[x_{1}, \dots, x_{i}]\) for \(j = 0, \dots, \ell\). Again, if \(\ell = 0\), then \(f = d_{0}\) and, by induction, \(\lm(f)\) does not contain \(x_{k}\). Suppose that \(\ell \neq 0\). Note that, in this case,
    \begin{equation} \label{eq_nucPolDerTriMonLidXk_2}
        \lm(f) = x_{i+1}^{\ell}\lm(d_{\ell}).
    \end{equation}
    Applying \(D\) to \(f\), as we did in Equation \eqref{eq_nucPolDerTriMonLidXk_1}, we get
    \[
        \begin{split}
            0 &= D(f) = D(x_{i+1}^{\ell}d_{\ell} + x_{i+1}^{\ell-1}d_{\ell-1} + \cdots + x_{i+1}^{2}d_{2} + x_{i+1}d_{1} + d_{0}) = \\
            &= \ell x_{i+1}^{\ell-1}f_{i+1}d_{\ell} + x_{i+1}^{\ell}D(d_{\ell}) + (\ell-1)x_{i+1}^{\ell-2}f_{i+1}d_{\ell-1} + x_{i+1}^{\ell-1}D(d_{\ell-1}) +\cdots + \\
            &  + 2x_{i+1}f_{i+1}d_{2} + x_{i+1}^{2}D(d_{2}) + f_{i+1}d_{1} + x_{i+1}D(d_{1}) + D(d_{0}),
        \end{split}
    \]
    which gives us
    \[
        D(d_{\ell})x_{i+1}^{\ell} + \big(\ell f_{i+1}d_{\ell} + D(d_{\ell-1})\big)x_{i+1}^{\ell-1} + \cdots + \big(2f_{i+1}d_{2} + D(d_{1})\big)x_{i+1} + f_{i+1}d_{1} + D(d_{0}) = 0.
    \]
    From this, it follows that \(D(d_{\ell}) = 0\). By induction, \(\lm(d_{\ell})\) does not contain \(x_{k}\). Therefore, by Equation \eqref{eq_nucPolDerTriMonLidXk_2}, the same holds for \(\lm(f)\).
    %\hfill\(\boxtimes\)
    \end{proof}

%%%%%%%%%%%%%%%%%%%%%%%%%%%%%%%%%
%
%       DERIVAÇÕES TRIANGULARES - APLICAÇÃO DO MÉTODO DAS CARACTERÍSTICAS
%
%%%%%%%%%%%%%%%%%%%%%%%%%%%%%%%%%
\subsection{Application of the method of integral curves}
\label{derivTriang_aplMetCar}

Our interest in triangular derivations is motivated by determining the algebraically independent generators of the rational invariant field \(\bbF(L)^{L}\) of a nilpotent Lie algebra \(L\), 
as we will see in Section \ref{sec:inv_nilp_lie}. 
We will show that, despite the difficulty of identifying a triangular derivation, when we establish an appropriate basis 
for the nilpotent Lie algebra, all the derivations necessary to determine the algebraically independent generators of \(\bbF(L)^{L}\) will be 
triangular derivations with respect to certain sets.

Suppose, in this section, that $f_k,f_{k+1},\ldots,f_n\in\bbF[x_1,\ldots,x_n]$ are polynomials such that $f_k\neq 0$ and let
    \[
        D = f_{k}D_{x_{k}} + \cdots + f_{n}D_{x_{n}} \in \derfrpolx
    \]
    be a triangular derivation. In particular, $f_i\in\bbF[x_1,\ldots,x_{i-1}]$ for all $i$. 
    We set $f_1,\ldots,f_{k-1}=0$. 
    Let $\bm F$ denote the corresponding vector field.  
Suppose that $\bfc=(c_1,\ldots,c_n):\bbF \dashrightarrow \bbF^n$ is a rational parametrization of an integral curve for $\bm F$. Then solutions $\bfc$ for the following 
system of ODEs are integral curves: 
\[ 
D_tc_i=f_i(c_1,\ldots,c_{i-1})\quad\mbox{for}\quad i=1,\ldots,n.
\] 
Let  $\int f\,d\!t$ denote the unique indefinite integral with constant term zero. 
Since each $f_i$ is a polynomial function in the variables $x_{1},\ldots,x_{i-1}$, we may define 
\begin{equation}\label{eq:genic1} 
C_1=x_1,\ldots,C_{k-1}=x_{k-1},\quad C_k=\int f_k d\!t + x_k = f_kt+x_{k}
\end{equation}
and, for $i=k+1,\ldots,n$, recursively, 
\begin{equation}\label{eq:genic2}
C_{i}=\int f_{i}(C_1,\ldots,C_{i-1})\,d\!t +x_{i}.
\end{equation}
We have that $C_i\in\bbF[x_1,\ldots,x_n][t]$. Set 
\begin{equation}\label{eq:genic}
\bfc=(C_1,\ldots,C_n).
\end{equation} 
We call $\bfc$ the {\em generic integral curve} for the vector field $\bm F$. 
If we take
$\bfp\in\bbF^n$ and substitute $\bfp$ in $\bfc$, then we obtain an {\em actual integral curve} 
$\bfc(\bfp)$ such that 
$\bfc(\bfp)(0)=\bfp$. 

The following lemma is routine to verify.

\begin{lemma}\label{lem:ic_facts} 
    Suppose that $\bfc=(C_1,\ldots,C_n)\in\bbF[x_1,\ldots,x_n,t]^n$ is as defined in the 
    previous paragraph and let $U=\bbF^n\setminus \calnV(f_k)$. 
\begin{enumerate}
    \item[(i)] If $\bfp\in\bbF^n$, then $\bfp$ is the unique point $\bfx$ in $\bbF^n$ 
    such that $\bfc(\bfx)(0)=\bfp$. 
    \item[(ii)] For $\bfp\in U$, $\bfc(\bfp)$ meets the hyperplane $\calnV(x_k)$ at $t=-p_k/f_k(p_1,\ldots,p_{k-1})$.
    \item[(iii)] If $\bfq,\bfp\in\bbF^n$ and $t_\bfq\in \bbF$, such that $\bfc(\bfp)(t_{\bfq})=\bfq$, then 
    $\bfc(\bfq)=\bfc(\bfp)(t+t_\bfq)$. 
    \item[(iv)] If $\bfp\in U$, then $\bfy=\bfc(\bfp)(-p_k/f_k(p_1,\ldots,p_{k-1}))$ is the unique 
    point $\bfy\in\calnV(x_k)$ such that $\bfp$ lies on $\bfc(\bfy)$. In fact $\bfc(\bfy)(p_k/f_k(p_1,\ldots,p_{k-1}))=\bfp$. 
\end{enumerate}
\end{lemma}

The statements of Lemma~\ref{lem:ic_facts} imply that triangular derivations satisfy Condition IC.
This will allow us to determine the algebraically independent generators of the kernel of a triangular derivation. 

\begin{theorem} \label{teor_derLocNilSatH1H2}
    Let $f_k,f_{k+1},\ldots,f_n\in\bbF[x_1,\ldots,x_n]$ be polynomials such that $f_k\neq 0$, let
    \[
        D = f_{k}D_{x_{k}} + \cdots + f_{n}D_{x_{n}} \in \derfrpolx
    \]
    be a triangular derivation and set $U=\bbF^n\setminus\calnV(f_k)$.  
    Let us define 
    \[ 
    \chi_0:U \to \calnV(x_k),\quad \bfp=(p_1,\ldots,p_n)\mapsto \bfc(\bfp)(-p_k/f_k(p_1,\ldots,p_{k-1})).
    \] 
    Then  $\chi_0:\bbF^n\dashrightarrow\calnV(x_k)$ is a rational map such that 
    $(\chi_0)|_{\calnV(x_k)\cap U}$ is the identity. 
    Furthermore, 
    for all $f\in \frpolx^D$, we have that $\mbox{dom}(f)\cap \calnV(x_k)\neq \emptyset$.
\end{theorem}
\begin{proof}
    The map $\chi_0$ is clearly a rational map, since the components of $\bfc$ are polynomial functions. 
    If $\bfy\in\mbox{Im}(\chi_0)$, then $\chi_0(\bfp)=\bfy$ if and only if $\bfp\in \mbox{Im}(\bfc(\bfy))$, and 
    so the fiber $\chi_0^{-1}(\bfy)$ is the image inside $U$ of the integral curve $\bfc(\bfy)$. 
    If $\bfp\in U\cap \calnV(x_k)$, then $\chi_0(\bfp)=\bfc(\bfp)(0)=\bfp$. Thus $(\chi_0)|_{\calnV(x_k)\cap U}$ is 
    the identity map on ${\calnV(x_k)\cap U}$.

    It remains to show the last statement. Suppose that $f=g/h$ where $\mbox{gcd}(g,h)=1$ such that $D(f)=0$. 
    Suppose by contradiction that $\mbox{dom}(f)\cap \calnV(x_k)=
    \emptyset$ which amounts to assuming that $x_k\mid h$. Write $h=x_k^\alpha h_0$ with $\alpha\geq 1$ and 
    $x_k\nmid h_0$. Then 
    \[ 
    0=D(f)=D(g/h)=(D(g)h-gD(h))/h^2
    \] which gives that $D(g)h-gD(h)=0$; that is 
    $D(g)h=gD(h)$. Since $\mbox{gcd}(g,h)=1$, this implies that $h\mid D(h)$ and in particular 
    $x_k^\alpha \mid D(h)$. On the other hand 
    \[ 
    D(h)=D(x_k^\alpha h_0)=\alpha x_{k}^{\alpha-1}D(x_k)h_0+x_k^\alpha D(h_0)=\alpha x_{k}^{\alpha-1}f_kh_0+x_k^\alpha D(h_0)
    \]  
and, in particular, $x_k\mid f_k$. However, this is impossible, since $f_k\in\bbF[x_1,\ldots,x_{k-1}]$. 
This shows that $\mbox{dom}(f)\cap \calnV(x_k)\neq 
    \emptyset$, as required.
\end{proof}

\begin{corollary} \label{coro_derLocNilSatH1H2}
    Using the notation of Theorem~\ref{teor_derLocNilSatH1H2} and letting $\psi:\bbF^{n-1}\to\bbF^n$ 
    be the obvious bijection onto $\calnV(x_k)$, the map $\chi=\psi^{-1}\circ\chi_0:\bbF^n\dashrightarrow \bbF^{n-1}$ satisfies Condition IC. Furthermore, 
    $\chi\circ \psi$ is the identity map on $\psi^{-1}(U)$ and for all $f\in \frpolx^D$, we have that 
    $\mbox{dom}(f)\cap \mbox{Im}(\psi)\neq \emptyset$. Consequently,
    the components of $\bfc(-x_k/f_k)$ other than the $k$-th form an algebraically independent generating set for 
    the field $\frpolx^D$. More concretely,
    define 
    \[ 
    z_i=x_i\quad\mbox{for}\quad i=1,\ldots,k-1
    \] 
    and
    \[ 
    z_i=C_i(-x_k/f_k)\quad\mbox{for}\quad i=k+1,\ldots,n.
    \] 
    Then $z_i\in\frpolx$ for all $i$ and $z_1,\ldots,z_{k-1},z_{k+1},\ldots,z_n$
    is an algebraically independent set of generators for $\frpolx^D$.  
\end{corollary}
\begin{proof}
    Composing the properties of $\chi_0$ from Theorem~\ref{teor_derLocNilSatH1H2} with the bijection $\psi$ shows that $\chi=\psi^{-1}\circ\chi_0$ satisfies Condition IC together with the additional hypothesis of Proposition~\ref{prop_isoFGammaKerD}, so that the conditions of Corollary~\ref{coro_isoFGammaKerD} hold,
    and so we need to compute $\chi^*(x_i)=x_i\circ \chi$ for $i=1,\ldots,n-1$.
    The element $z_i=x_i\circ \chi$ is the $i$-th component of $\chi:\bbF^n\dashrightarrow\bbF^{n-1}$.
    Given that $\chi=\psi^{-1}\circ \chi_0$, we have that the $i$-th component is $C_i(-x_k/f_k)$. Hence 
    the description for the $z_i$ follows from~\eqref{eq:genic1} and~\eqref{eq:genic2}.  
\end{proof}

The next result shows that the generators obtained by the method of integral curves are the same 
as the ones provided by the Dixmier map~\cite[Section~1.1.9]{Fre17}.

\begin{theorem}\label{th:dix}
    Suppose that $z_1,\ldots,z_{k-1},z_{k+1},\ldots,z_n$ are the generators of 
    $\frpolx^D$ in Corollary~\ref{coro_derLocNilSatH1H2}. Then, for $i\neq k$, we have 
    \begin{equation}\label{eq:dix}
    z_i=\sum_{m\geq 0}\frac{ (-1)^m}{m!}D^m(x_i)\left(\frac{x_k}{f_k}\right)^m.
    \end{equation}
    The sum is finite, since $D$ is locally nilpotent.
\end{theorem}
\begin{proof}
    For $i=1,\ldots,n$, let $C_i\in \bbF[x_1,\ldots,x_n,t]$ be as defined before 
    Lemma~\ref{lem:ic_facts}. Then $C_i$ is  
    the $i$-th component of a generic integral curve $\bfc$ and such an integral curve satisfies 
    the equation $D_t(C_i)=f_i(C_1,\ldots,C_{i-1})$. Also note that 
    $\bfc(0)=(x_1,\ldots,x_n)$. 
We claim that
\begin{equation}\label{eq:dtm}
    D_t^mC_i=D_t^m\left(x_i \circ \bfc\right) = D^m(x_i) \circ \bfc
\end{equation}
for all $m \geq 0$. For $m = 1$, the chain rule gives
\[
    D_t\bigl(x_i \circ \bfc\bigr)
    = \sum_{j=1}^{n} \bigl(D_{x_j}(x_i) \circ \bfc\bigr) \cdot D_t(C_j)
    = \sum_{j=1}^{n} \bigl(D_{x_j}(x_i) \circ \bfc\bigr) \cdot (f_j \circ \bfc)
    = \Bigl(\sum_{j=1}^{n} f_j D_{x_j}(x_i)\Bigr) \circ \bfc
    = D(x_i) \circ \bfc.
\]
Applying the same argument inductively to $D^{m-1}(x_i)$ in place of $x_i$ establishes \eqref{eq:dtm} for all $m$.
Evaluating at $t = 0$ and using $\bfc(0) = (x_1,\ldots,x_n)$ gives
\[
    \left(D_t^mC_i\right)_{t=0} = D^m(x_i),
\]
so the Taylor series of the coordinate $C_i$ is
\[
    C_i(t) = \sum_{m \geq 0} \frac{t^m}{m!}\,D^m(x_i).
\]
Evaluating at $t = -x_k/f_k$,  we obtain
\[
    C_i\left(-\frac{x_k}{f_k}\right)
    = \sum_{m \geq 0} \frac{(-1)^m}{m!}\,D^m(x_i)\,\frac{x_k^m}{f_k^m}.
\]
Since the $i$-th generator $z_i$ is given precisely by $C_i(-x_k/f_k)$, the proof is complete.
\end{proof}

The proof admits a flow-theoretic interpretation. The generic integral curve
\[
\bfc=(C_1,\ldots,C_n)
\]
defines a one-parameter family of rational maps
\[
\phi_t:\bbF^n\to \bbF^n,\qquad \phi_t(\bfx)=\bfc(\bfx)(t),
\]
which is the flow of the vector field associated with \(D\). Thus, for each \(\bfp\in U\), the point \(\phi_t(\bfp)\) is obtained by moving \(\bfp\) along its integral curve for time \(t\). If we set
\[
\tau=-\frac{x_k}{f_k},
\]
then \(\tau(\bfp)\) is precisely the unique time at which the integral curve through \(\bfp\) meets the hyperplane \(\calnV(x_k)\). In other words,
\[
\tau(\phi_t(\bfp))=\tau(\bfp)-t.
\]
Applying $D$ to $\tau$, we obtain
\[
D(\tau)=-1.
\]
Hence \(-\tau=x_k/f_k\) is a slice for \(D\), and the substitution \(t=-x_k/f_k\) in the Taylor expansion of \(C_i(t)\) is exactly the evaluation of the integral curve at its hitting time on the hypersurface \(x_k=0\).

The formula~\eqref{eq:dix} shows that the generators obtained by the method of integral curves coincide with 
the generators given by the Dixmier map. The function $x_k$ is said to be a {\em local slice} for the 
derivation $D$ in Theorem~\ref{teor_derLocNilSatH1H2}, since $D^2(x_k)=0$ and the function $x_k/f_k$ is 
said to be a {\em slice}, as $D(x_k/f_k)=1$. The Dixmier map $\delta_s:\bbF[\bm x_n]\to \bbF[\bm x_n,f_k^{-1}]$ 
for the slice $s=x_k/f_k$ is defined, for $a\in \polx$, as 
\[
\delta_s(a)=\sum_{m\geq 0}\frac{(-1)^m}{m!}D^m(a)s^m=\sum_{m\geq 0}\frac{(-1)^m}{m!}D^m(a)\left(\frac{x_k}{f_k}\right)^m.
\]  
As is known, the $\delta_s(x_i)$ with $i\neq k$ form an algebraically independent system of generators for $\frpolx^D$ (see~\cite[Principle~11, page~27]{Fre17}) and Theorem~\ref{th:dix} implies that this generating system agrees with 
the system obtained by the method of integral curves. In fact, the proof of Theorem~\ref{th:dix} also shows
that $\delta_s(x_i)$ is the Taylor expansion of the $i$-th component $C_i$ of a generic integral curve $\bfc$.

\section{Polynomial invariants}
\label{derivTriang_invPol}
Note that the generators $z_1,\ldots,z_{k-1},z_{k+1},\ldots,z_n$ of the invariant field $\frpolx^D$ in Corollary~\ref{coro_derLocNilSatH1H2} lie in 
the localization $\bbF[x_1,\ldots,x_n,f_k^{-1}]$. Since $f_k\in \frpolx^D$, one can multiply the generators by a suitable power of $f_k$ and obtain a polynomial generating set of the form 
\[ 
u_i=z_i=x_i\quad\mbox{for}\quad i\in \{1,\ldots,k-1\}
\] 
and 
\[ 
u_i=f_k^{\alpha_i}z_i=f_k^{\alpha_i}x_i+\widetilde C_{i} \quad\mbox{for}\quad i\in\{k+1,\ldots,n\}
\] 
where $\widetilde C_i\in\bbF[x_1,\ldots,x_{i-1}]$. 
\begin{theorem}\label{teor_nucPolDerTriIncGerFrac}
Using the notation of Theorem~\ref{teor_derLocNilSatH1H2} and the previous paragraph, the following 
are valid:
\begin{enumerate}
\item $\frpolx^D=\bbF(u_1,\ldots,u_{k-1},u_{k+1},\ldots,u_n)$;
\item the $u_i$ are algebraically independent;
\item $\polx^D\subseteq \bbF[u_1,\ldots,u_{k-1},u_{k+1},\ldots,u_n,f_k^{-1}]$.
\end{enumerate}
\end{theorem}
\begin{proof}
    Statement (1) follows at once from the fact that $f_k\in\polx^D$, while statement (2) follows from the 
    fact that the $z_i$ are algebraically independent and $f_k\in \bbF[x_1,\ldots,x_{k-1}]$. 

(3)    Let \(h \in \polx^{D}\). Note that
    \begin{align*}
        \bbF[u_{1}, \dots, u_{k-1},u_{k+1},\ldots,u_{n}, f_{k}^{-1}] &= \bbF[u_{1}, \dots, u_{k-1}][u_{k+1}, \dots, u_{n}, f_{k}^{-1}] \\&= \bbF[x_{1}, \dots, x_{k-1}][u_{k+1}, \dots, u_{n}, f_{k}^{-1}].
    \end{align*}
    Thus, denoting \(R \coloneqq \bbF[x_{1}, \dots, x_{k-1}]\), we need to show that \(h \in R[u_{k+1}, \dots, u_{n}, f_{k}^{-1}]\). For this, we will use induction on \(\lm(h)\), where we use the lexicographic monomial 
    order on $R[x_k,\ldots,x_n]$
    with respect to \(x_{k} < \cdots < x_{n}\). Thus, if \(\lm(h) = 1\), then \(h \in R\) and the theorem's statement is true. Suppose that \(\lm(h) \neq 1\) and that the theorem's statement holds for any polynomial in \(\polx^{D}\) with a leading monomial smaller than \(\lm(h)\). By Lemma \ref{lema_nucPolDerTriMonLidXk}, we know that \(\lm(h)\) does not contain the variable \(x_{k}\). Thus, write the leading term of \(h\) as
    \[
        \lt(h) = \alpha x_{k+1}^{\beta_{k+1}}\cdots x_{n}^{\beta_{n}},
    \]
    where \(\alpha \in R\). Define
    \begin{align*}
        f &\coloneqq \alpha u_{k+1}^{\beta_{k+1}}\cdots u_{n}^{\beta_{n}} = \alpha(f_{k}^{\alpha_{k+1}}x_{k+1} + \widetilde C_{k+1})^{\beta_{k+1}} \cdots (f_{k}^{\alpha_{n}}x_{n} + \widetilde C_{n})^{\beta_{n}}\text{ and}\\ \beta &\coloneqq \beta_{k+1}\alpha_{k+1} + \cdots + \beta_{n}\alpha_{n}.
    \end{align*}
    Then, since \(\widetilde C_{i}\in\bbF[x_1,\ldots,x_{i-1}]\) does not involve \(x_i,\ldots,x_n\), for all \(i = k+1, \dots, n\), the leading term of \(f\) is 
    \[
        \lt(f) = \alpha f_{k}^{\beta}x_{k+1}^{\beta_{k+1}}\cdots x_{n}^{\beta_{n}}.
    \]
    Note that \(\alpha f_{k}^{\beta} \in R\). Thus, in the algebra \(R[x_{k+1}, \dots, x_{n}]\), we have that \(\lm(f) = \lm(h)\). Define \(h_{0} \coloneqq f_{k}^{\beta}h - f\). Then 
    $\lm(h_{0}) <  \lm(h)$. Note that, by construction, \(D(f) = 0\). Also note that, since \(h \in \polx^{D}\) and \(f_{k} \in \bbF[x_{1}, \dots, x_{k-1}]\), then \(D(f_{k}^{\beta}h) = 0\). Thus, \(h_{0} \in \polx^{D}\). Therefore, by induction, \(h_{0} \in R[u_{k+1}, \dots, u_{n}, f_{k}^{-1}]\). Hence, since \(f \in R[u_{k+1}, \dots, u_{n}, f_{k}^{-1}]\), we have that 
    \[
        h = (h_{0} + f)f_{k}^{-\beta} \in R[u_{k+1}, \dots, u_{n}, f_{k}^{-1}],
    \]
    which completes the proof.
    %\hfill\(\boxtimes\)
\end{proof}

\section{Invariants of nilpotent Lie algebras}
\label{gerAlgInvarRac_algoritmoLieNilp}
\label{sec:inv_nilp_lie}

Let $L$ be a finite-dimensional Lie algebra over a field $\bbF$ of characteristic zero. 
Its symmetric algebra $\bbF[L]$ may be identified with the algebra of polynomial functions on 
the dual space $L^{*}$: each $x\in L$ defines a linear function on $L^{*}$ by 
$\xi\mapsto \xi(x)$, and, by the universal property of the symmetric algebra, 
this extends uniquely to an algebra isomorphism between $\bbF[L]$ and the polynomial 
algebra on $L^{*}$. Hence, after choosing a basis $\{x_{1},\dots,x_{n}\}$ of $L$, 
$x_i$ can be viewed as the $i$-th coordinate function on $L^*$ with respect to the dual basis, and so 
we may regard $\bbF[L]$ as $\bbF[x_{1},\dots,x_{n}]$. 
Moreover, the adjoint representation of $L$ on itself extends uniquely to an action on 
$\bbF[L]$ by derivations: for each $x\in L$, the operator $\ad(x)$ is determined by its 
values on degree-one elements, namely $\ad(x)(y)=[x,y]$ for $y\in L$, and by the 
Leibniz rule on products. Passing to the field of fractions $\bbF(L)=\bbF(x_1,\ldots,x_n)$, this yields 
derivations still denoted by $\ad(x)$, and the rational invariants of $L$ are precisely 
the elements annihilated by all these derivations.
Hence the field of rational invariants of a Lie algebra \(L\) is
\[
    \bbF(L)^{L} \coloneqq \{f \in \bbF(L) \mid \ad(x)(f) = 0, \,\, \forall \; x \in L\} = \bigcap_{i=1}^{n}\frpolx^{\ad(x_{i})}.
\]

Suppose that $L$ is a nilpotent Lie algebra as above and suppose that 
$x_1,\ldots,x_n$ is a basis such that for all $1\leq i<j\leq n$, 
\[ 
[x_i,x_j]=\sum_{k=1}^{i-1}\alpha_{i,j}^kx_k
\] 
with some $\alpha_{i,j}^k\in \bbF$. Such a basis is said to be a {\em triangular basis} for $L$.

\begin{lemma}\label{lem:trbasis}
Suppose that $x_1,\ldots,x_n$ is a triangular basis for a nilpotent Lie algebra $L$ and suppose for some $k\leq n$ that 
\[ 
K=\{f\in \bbF(L)\mid \ad(x_i)(f)=0\mbox{ for $i=1,\ldots,k$}\}.
\] 
Then $K$ is invariant under $\ad(x)$ for all $x\in L$. 
\end{lemma}
\begin{proof}
Suppose that $f \in K$ and let $x\in L$. We need to show that $\ad(x_i)\ad(x)(f)=0$ for all 
$i=1,\ldots,k$. Let $i\in\{1,\ldots,k\}$ and note that 
\[ 
\ad(x_i)\ad(x)=[\ad(x_i),\ad(x)]+\ad(x)\ad(x_i)=
\ad([x_i,x])+\ad(x)\ad(x_i).
\] 
Since the chosen basis is triangular, $[x_i,x]$ is a linear combination $\sum_{j=1}^{i-1}\alpha_j x_j$, and so $\ad([x_i,x])(f)=0$. Since $\ad(x_i)(f)=0$, by assumption, 
$\ad(x_i)\ad(x)(f)=0$, and hence $\ad(x)(f)\in K$. 
\end{proof}

Lemma~\ref{lem:trbasis} suggests a practical procedure to compute generators of the invariant field 
of $L$. We start by setting $K_0=\frpolx$ and define, for $i=1,\ldots,n$,
\begin{align}\label{eq:Ki}
K_i&=\{f\in \frpolx\mid \ad(x_j)(f)=0\mbox{ for all $j=1,\ldots,i$}\}\\&=
\{f\in K_{i-1}\mid \ad(x_i)(f)=0\}\notag.
\end{align}
As long as the restriction of $\ad(x_i)$ to $K_{i-1}$ is a triangular derivation, 
the procedure in Section~\ref{sec:triang} can be used to compute the kernel 
$K_{i-1}^{\ad(x_i)}$ of the restriction of $\ad(x_i)$ to $K_{i-1}$. This way, we obtain 
\[ 
\bbF(L)^L=K_n. 
\] 
The following theorem shows that this assumption is, in fact, valid. 
A system of polynomials $z_1,\ldots,z_m\in\bbF[x_1,\ldots,x_n]$ is said to be {\em triangular} if 
there is a sequence of indices $a_1<\cdots < a_m$ such that, for all $i$,
\[ 
z_i=x_{a_i}q_{i,1}+q_{i,2}\quad\mbox{where}\quad q_{i,1},q_{i,2}\in\bbF[x_1,\ldots,x_{a_i-1}] \mbox{ and } q_{i,1}\neq 0. 
\] 
A triangular system is clearly algebraically independent.

\begin{theorem} \label{teor_presDerTriangAd}
    Let us suppose that \(\{x_{1}, \dots, x_{n}\}\) is a triangular basis of a nilpotent Lie algebra \(L\), set 
    $K_0=\bbF(L)=\frpolx$ and, for $i\in\{1,\ldots,n\}$ let $K_i$ 
    be defined as in~\eqref{eq:Ki}. 
    Then, for each \(i = 1, \dots, n\), the following is valid.
    \begin{enumerate}[(a)]
        \item The field \(K_{i}\) is generated by a triangular system of polynomials \(\{z_{i,1}, \dots, z_{i,m_{i}}\} \subseteq \polx\). More precisely,  there is a 
        sequence of indices $1\leq a_{i,1}<\cdots <a_{i,m_i}\leq n$ such that, for each $j\in \{1,\ldots,m_i\}$,
        \[ 
        z_{i,j}=x_{a_{i,j}}q_{i,j,1}+q_{i,j,2}
        \] 
        such that $q_{i,j,1},q_{i,j,2}\in \bbF[x_1,\ldots,x_{a_{i,j}-1}]$ and $q_{i,j,1}\in \bbF[L]^L\setminus\{0\}$.  
        \item The set \(\{z_{i,1}, \dots, z_{i,m_{i}}\} \subseteq \polx\) is algebraically independent.
        \item For all \(j = 1, \dots, n\), \(\ad(x_{j})|_{K_{i}}\) is a triangular derivation with respect to the set \(\{z_{i,1}, \dots, z_{i,m_{i}}\}\) such that \(\ad(x_{j})(z_{i,1}) = 0\).
        \item There exists \(d_{i} \in \bbF[L]^L\cap \bbF[z_{i,1}, \dots, z_{i,m_{i}}]\) such that \(K_{i} \cap \polx \subseteq \bbF[z_{i,1}, \dots, z_{i,m_{i}}, d_{i}^{-1}]\).
    \end{enumerate}
\end{theorem}
\begin{proof}
    We will prove this theorem using induction on \(i\). If \(i = 1\), then \(K_{i} = K_{1} = \bbF(L) = \frpolx\) (since $x_1$ is centeral in $L$) and statements (a)--(d) hold for the generating set $x_1,\ldots,x_n$ and 
    with $d_0=1$.

    Suppose, by induction, that the statement holds for some \(i \in \{1, \dots, n-1\}\).
    To simplify the notation, we omit the index $i$ from the data that is available by the induction hypothesis. In other words, we denote the triangular system \(z_{i,1}, \dots, z_{i,m_{i}}\) 
    simply by $z_1,\ldots,z_m$ and the restriction of $\ad(x_{i+1})$ to $K_i$ by $D$. 
    By assumption, there is an index sequence \(1 \le a_{1} < \cdots < a_{m} \le n\)  such that 
    for all \(j = 1, \dots, m\),
        \[
            z_{j} = x_{a_{j}}q_{j,1} + q_{j,2},
        \]
        with \(q_{j,1}, q_{j,2} \in \bbF[x_{1}, \dots, x_{a_{j}-1}]\), $q_{j,1}\in\bbF[L]^L$ and $q_{j,1} \neq 0$.
    By the induction hypothesis, \(D\) is a triangular derivation with respect to the set \(\{z_{1}, \dots, z_{m}\}\). If $D=0$, then $K_{i+1}=K_i$ and taking the generators $z_1,\ldots,z_m$ for $K_{i+1}$,  
    the statements of the theorem remain true for $i+1$. If $D\neq 0$, then  
    \[
        D = D(z_k)D_{z_k}+D(z_{k+1})D_{z_{k+1}}+\cdots+D(z_m)D_{z_m}=f_{k}D_{z_{k}} + f_{k+1}D_{z_{k+1}}+\cdots + f_{m}D_{z_{m}},
    \]
    where \(f_{k} \neq 0\) and \(f_{j} \in \bbF[z_{1}, \dots, z_{j-1}]\), for all \(j = k, \dots, m\). 
    It is also true by the induction hypothesis that $D(z_1)=0$ and hence $k\geq 2$. By Theorem~\ref{teor_nucPolDerTriIncGerFrac},  the kernel \(K_{i+1}\) of $D$ has an algebraically independent generating system  of 
    the following form:
    \[
        \begin{array}{cccccc}
            u_{1} = z_{1}, &
            \ldots, &
            u_{k-1}  = z_{k-1}, &
            u_{k+1} = f_{k}^{\alpha_{k+1}}z_{k+1} + g_{k+1}, &
            \ldots, &
            u_{m} = f_{k}^{\alpha_{m}}z_{m} + g_{m},
        \end{array}
    \]
    where \(g_{j} \in \bbF[z_{1}, \dots, z_{j-1}]\), for \(j = k+1, \dots, m\). 
        Note that $u_k$ is missing from the list $u_s$ and this will be indicated by $\cancel{u_k}$ when 
    we list these generators. 
Writing out these generators, we have that 
    \begin{align*}
    u_j&=z_j=x_{a_j}q_{j,1}+q_{j,2}\mbox{ for } j=1,\ldots,k-1\\
    u_j&=f_{k}^{\alpha_{j}}z_{j} + g_{j}=x_{a_j}f_k^{\alpha_j}q_{j,1}+f_k^{\alpha_j}q_{j,2}+g_j 
    \mbox{ for } j=k+1,\ldots,m.
    \end{align*}
    Since \(z_{j} \in \bbF[x_{1}, \dots, x_{a_{j}}]\), it follows that 
    \[ 
    f_{k}\in 
    \bbF[z_1,\ldots,z_{k-1}]\subseteq \bbF[x_{1}, \dots, x_{a_{k-1}}]\quad\mbox{and}\quad 
    g_s\in\bbF[z_1,\ldots,z_{s-1}]\subseteq \bbF[x_1,\ldots,x_{a_s-1}],
    \]  for all \(s = k+1, \dots, m\).
    Thus the system \(\{u_{1}, \dots, 
    \cancel{u_{k}},\ldots,u_{m}\} \subseteq \polx\) is triangular with respect to 
    the index sequence  
    \[ 
    1 \le a_{1} < \cdots < \cancel{a_{k}} < \cdots < a_{m} \le n.
    \] 
    In particular, the generating system $u_s$ is algebraically independent. 

    In order to conclude that the generating system $u_s$ satisfies the 
    conditions (a) and (b) of the theorem, we need to show that the coefficient 
    of $x_{a_s}$ in $u_s$ is invariant under $\ad (y)$ for all $y\in L$. For $s\leq k-1$, 
    this follows from the induction hypothesis. For $s\geq k+1$, this follows, once 
    we prove that $\ad(y)(f_k)=0$ for all $y\in L$. 
    Let us now verify this claim.  In fact, $f_k=D(z_k)=\ad(x_{i+1})(z_k)$ 
    and so 
    \[ 
        \ad(y)(f_k)=\ad(y)\ad(x_{i+1})(z_k)=\ad(x_{i+1})\ad(y)(z_k)-\ad([x_{i+1},y])(z_k).
    \] 
    Now, by the induction hypothesis, $\ad(y)$ is a triangular derivation with respect to the 
    generating system $z_1,\ldots,z_m$, and so $\ad(y)(z_k)\in\bbF[z_1,\ldots,z_{k-1}]$, which gives that the first term $\ad(x_{i+1})\ad(y)(z_k)=D(\ad(y)(z_k))$ 
    of the last displayed equation is zero. Since the given basis of $L$ is triangular, we also obtain 
    that $[x_{i+1},y]$ is a linear combination of $x_1,\ldots,x_i$ which shows that the second term is also 
    equal to zero. Hence $\ad(y)(f_k)=0$. Therefore, we have that the set \(\{u_{1}, \ldots,\cancel{u_k},\dots, u_{m}\}\) satisfies conditions (a) and (b) of the theorem.
    If \(j = 1, \dots, k-1\), then, by the induction hypothesis, it holds, for all 
    \(\ell = 1, \dots, n\), that \(\ad(x_{\ell})(u_{j}) \in \bbF[u_{1}, \dots, u_{j-1}]\), because, in this case, \(u_{j} = z_{j}\). However, this property does not necessarily hold if \(j \ge k+1\), possibly requiring a small alteration in the set \(\{u_{k+1}, \dots, u_{m}\}\). 
    Define $u_1'=u_1,\ldots,u_{k-1}'=u_{k-1}$.
    We claim that there exist $\beta_{k+1},\ldots,\beta_m\in\bbZ_{\ge 0}$ such that, setting 
    $u_s'=f_k^{\beta_s}u_s$ for all $s\in\{k+1,\ldots,m\}$, the modified generating set $\{u_1',\ldots,\cancel{u_k'},\ldots
    u'_m\}$ satisfies statements (a), (b), and (c) of the theorem. 
    Since $f_k\in\bbF[z_1,\ldots,z_{k-1}]\subseteq\bbF[x_1,\ldots,x_{a_{k-1}}]$, we have that the generators $u_i'$ satisfy 
    statements (a) and (b), since the generators $u_i$ also satisfy the same conditions. 
    So we need only be concerned with statement (c).

     Assume by induction that, for some 
    $r\in\{k-1,k+1,\ldots,m-1\}$, we have constructed the generating set $u_1',\ldots,u_{k-1}',u_{k+1}',\ldots,u_r',u_{r+1},\ldots,u_m$ for $K_{i+1}$ that satisfies 
    conditions (a)--(b) of the theorem and also $\ad(x_\ell)(u_j')\in\bbF[u_1',\ldots,\cancel{u_k},\ldots,u_{j-1}']$ for 
    $\ell=1,\ldots,n$ and for $j=1,\ldots,r$.
Let $\ell\in\{1,\ldots,n\}$. Since \(\ad(x_{\ell})\) is a triangular derivation with respect to the set \(\{z_{1}, \dots, z_{m}\}\), \(\ad(x_{\ell})(f_{k}) = 0\) and \(g_{r+1} \in \bbF[z_{1}, \dots, z_{r}]\), we have that
    \begin{align} \label{eq_presDerTriangAd_adZ}
        \ad(x_{\ell})(u_{r+1}) &= \ad(x_{\ell})(f_{k}^{\alpha_{r+1}}z_{r+1} + g_{r+1})  \notag \\
        &= f_{k}^{\alpha_{r+1}}\ad(x_{\ell})(z_{r+1}) + \ad(x_{\ell})(g_{r+1}) \in \bbF[z_{1}, \dots, z_{r}].
    \end{align}
    Furthermore, 
    \[
        D\big(\ad(x_{\ell})(u_{r+1})\big) = \ad(x_{\ell})\big(D(u_{r+1})\big) - \ad\big([x_{\ell}, x_{i+1}]\big)(u_{r+1}) = 0,
    \]
    since \(u_{r+1} \in K_{i+1}\) and \([x_{\ell}, x_{i+1}]\) is a linear combination of \(x_{1}, \dots, x_{i}\). Therefore, by Theorem~\ref{teor_nucPolDerTriIncGerFrac}, there exist \(\beta_{r+1,\ell} \in \bbZ_{\ge 0}\) and \(h_{\ell} \in \bbF[t_{1}, \dots, t_{m-1}]\) such that
    \[
        \ad(x_{\ell})(u_{r+1}) = \dfrac{h_{\ell}(u_{1}', \dots,\cancel{u_k'},\ldots, u_{r}', u_{r+1}, \dots, u_{m})}{f_{k}^{\beta_{r+1,\ell}}} \in \bbF[u_{1}', \dots, \cancel{u_k'},\ldots,u_{r}', u_{r+1}, \dots, u_{m}, f_{k}^{-1}].
    \]
    From this, it follows that
    \[
        f_{k}^{\beta_{r+1,\ell}}\ad(x_{\ell})(u_{r+1}) = h_{\ell}(u_{1}', \dots,\cancel{u_k'},\ldots, u_{r}', u_{r+1}, \dots, u_{m}).
    \]
    Note that the left-hand side of the above equality is a polynomial in 
    \[
        \bbF[z_{1}, \dots, z_{r}] \subseteq \bbF[x_{1}, \dots, x_{a_{r}}].
    \]
    Since \(h_{\ell} \in \bbF[t_{1}, \dots, t_{m-1}]\) and since \(u_{j}\) contains the variable \(x_{a_{j}}\) in degree \(1\), but does not contain the variables \(x_{a_{j+1}}, \dots, x_{n}\), for all \(j = r+1, \dots, m\), then 
    \[ 
    h_{\ell}(u_{1}', \dots, \cancel{u_k'},\ldots,u_{r}', u_{r+1}, \dots, u_{m}) = h_{\ell}(u_{1}', \dots, \cancel{u_k'},\ldots,u_{r}').
    \]  
    Therefore,
    \begin{equation} \label{eq_presDerTriangAd_adFrac}
        \ad(x_{\ell})(u_{r+1}) = \dfrac{h_{\ell}(u_{1}', \dots, \cancel{u_k'},\ldots,u_{r}')}{f_{k}^{\beta_{r+1,\ell}}}.
    \end{equation}
    Define
    \[
        \beta_{r+1} \coloneqq \max(\beta_{r+1,1}, \dots, \beta_{r+1,n})
    \]
    and consider the set formed by the elements
    \[
        \begin{array}{ccccccccc}
            u_{1}', &
            \ldots, &
            \cancel{u_k'},&
            \ldots,&
            u_{r}', &
            u_{r+1}' \coloneqq f_{k}^{\beta_{r+1}}u_{r+1}, &
            u_{r+2}, &
            \ldots, &
            u_{m}.
        \end{array}
    \]
    %
%    Since \(K_{i+1} = \bbF(u_{1}, \dots, \cancel{u_k},\ldots,u_{m})\) and as \(f_{k} \in \bbF[u_{1}, \dots, u_{k-1}]\), then we have that 
 %   \[
 %   K_{i+1} = \bbF(u_{1}', \dots, \cancel{u_k'},\ldots,u_{r+1}', u_{r+2}, \dots, u_{m}).
 %   \]  
 %   Furthermore, since \(\{u_{1}, \dots, u_{m}\}\) is a triangular system of polynomials with respect to 
%    the index sequence \(1 \le a_{1} < \cdots < \cancel{a_{k}} < \cdots < a_{m} \le n\), the same is true for  the generating set \(\{u_{1}', \dots, \cancel{u_k'},\ldots,u_{r+1}', u_{r+2}, \dots, u_{m}\}\) and such a set is algebraically independent. Moreover, \(u_{r+1}' = f_{k}^{\beta_{r+1}}f_{k}^{\alpha_{r+1}}z_{r+1} + f_{k}^{\beta_{r+1}}g_{r+1}\), where \(f_{k}^{\beta_{r+1}}f_{k}^{\alpha_{r+1}} \in \bbF[L]^{L}\) and \(f_{k}^{\beta_{r+1}}f_{k}^{\alpha_{r+1}}, f_{k}^{\beta_{r+1}}g_{r} \in \bbF[x_{1}, \dots, x_{a_{r}}]\). Therefore, the set \(\{u_{1}', \dots, \cancel{u_k'},\ldots,u_{r+1}', u_{r+2}, \dots, u_{m_{i}-1}\}\) satisfies (a) and (b).
    %
 %   
    %
    Compute, for $\ell\in\{1,\ldots,n\}$, 
    \[
        \begin{split}
            \ad(x_{\ell})(u_{r+1}') &= \ad(x_{\ell})(f_{k}^{\beta_{r+1}}u_{r+1}) = f_{k}^{\beta_{r+1}}\ad(x_{\ell})(u_{r+1}) \overset{\mathclap{\substack{\text{\eqref{eq_presDerTriangAd_adFrac} \vspace{1pt}} \vspace{10pt} \mathstrut}}}{=}  f_{k}^{\beta_{r+1}}\frac{h_{\ell}(u_{1}', \dots, \cancel{u_k'},\ldots,u_{r}')}{f_{k}^{\beta_{r+1,\ell}}} = \\
            &= f_{k}^{\beta_{r+1}-\beta_{r+1,\ell}}h_{\ell}(u_{1}', \dots, \cancel{u_k'},\ldots,u_{r}') \in \bbF[u_{1}', \dots, \cancel{u_k'},\ldots,u_{r}'],
        \end{split}
    \]
    since \(\beta_{r+1}-\beta_{r+1,\ell} \ge 0\). Therefore, \(\ad(x_{\ell})(u_{j}') \in \bbF[u_{1}', \dots, \cancel{u_k'},\ldots,u_{j-1}']\), for all \(j = 1, \dots, r+1\) with $j\neq k$. 
    
    Performing this change of generator for each $r$, we obtain the generating system
    \[
        \begin{array}{cccccccc}
            z_{i+1,1}=u_{1}', &
            \ldots, &
            z_{i+1,k-1}=u_{k-1}', &
            z_{i+1,k}=u_{k+1}', &
            \ldots,&
            z_{i+1,m-1}=u_{m}' 
        \end{array}
    \]
    for $K_{i+1}$ that satisfies (a), (b), and (c). 
    
    Let us now turn to assertion~(d).
    For this, note that
    \begin{equation} \label{eq_presDerTriangAd_parteD}
        K_{i+1} \cap \bbF[L] = \big(K_{i} \cap \bbF[L]\big)^{\ad(x_{i+1})} \overset{\mathclap{\substack{\text{Induction} \\ \text{hypothesis} \vspace{5pt} \mathstrut}}}{\subseteq} \bbF[z_{1}, \dots, z_{m}, d_{i}^{-1}]^{\ad(x_{i+1})}.
    \end{equation}
    We will show that the following equality holds:
    \begin{equation} \label{eq_presDerTriangAd_invPolGer}
        \bbF[z_{1}, \dots, z_{m}, d_{i}^{-1}]^{\ad(x_{i+1})} = \bbF[z_{1}, \dots, z_{m}]^{\ad(x_{i+1})}[d_{i}^{-1}].
    \end{equation}
    Note that the inclusion \(\bbF[z_{1}, \dots, z_{m}]^{\ad(x_{i+1})}[d_{i}^{-1}] \subseteq \bbF[z_{1}, \dots, z_{m}, d_{i}^{-1}]^{\ad(x_{i+1})}\) is direct, since \(d_{i} \in \bbF[L]^{L}\). Now, let \(\dfrac{p(z_{1}, \dots, z_{m})}{d_{i}^{r}} \in \bbF[z_{1}, \dots, z_{m}, d_{i}^{-1}]^{\ad(x_{i+1})}\). Then
    \begin{equation}\label{eq_presDerTriangAd_numZero}
        0 = \ad(x_{i+1})\left(\dfrac{p(z_{1}, \dots, z_{m})}{d_{i}^{r}}\right) = \dfrac{\ad(x_{i+1})\big(p(z_{1}, \dots, z_{m})\big)d_{i}^{r}}{d_{i}^{2r}} = \dfrac{\ad(x_{i+1})\big(p(z_{1}, \dots, z_{m})\big)}{d_{i}^{r}},
    \end{equation}
    that is, \(\ad(x_{i+1})\big(p(z_{1}, \dots, z_{m})\big) = 0\). This implies the other inclusion, which demonstrates Equation \eqref{eq_presDerTriangAd_invPolGer}. Thus, returning to Equation \eqref{eq_presDerTriangAd_parteD}, we obtain
    \[
        K_{i+1} \cap \bbF[L] \subseteq \bbF[z_{1}, \dots, z_{m}]^{\ad(x_{i+1})}[d_{i}^{-1}].
    \]
    From Theorem \ref{teor_nucPolDerTriIncGerFrac}, we have that
    \[
        \bbF[z_{1}, \dots, z_{m}]^{\ad(x_{i+1})} \subseteq \bbF[u_{1}', \dots, \cancel{u_k'},\ldots,u_{m}', f_{k}^{-1}].
    \]
    Thus,
    \[
        K_{i+1} \cap \bbF[L] \subseteq \bbF[u_{1}', \dots, \cancel{u_k'},\ldots,u_{m}', f_{k}^{-1}, d_{i}^{-1}].
    \]
    By hypothesis, \(d_{i} \in \bbF[L]^{L} \cap \bbF[z_{1}, \dots, z_{m}]\), so \(d_{i} \in \bbF[z_{1}, \dots, z_{m}]^{\ad(x_{i+1})}\). Thus, by Theorem \ref{teor_nucPolDerTriIncGerFrac}, there exist a polynomial \(p\) in \(m-1\) variables and an \(s \in \bbZ_{\ge 0}\) such that
    \[
        d_{i} = \dfrac{p(u_{1}', \ldots,\cancel{u_k'},\dots, u_{m}')}{f_{k}^{s}} \in \bbF[u_{1}', \dots, \cancel{u_k'},\ldots,u_{m}', f_{k}^{-1}].
    \]
    Define
    \[
        d_{i+1} \coloneqq f_{k}p(u_{1}', \dots, \cancel{u_k'},\ldots,u_{m}').
    \]
    One can show, as in Equation \eqref{eq_presDerTriangAd_numZero}, that \(p(u_{1}', \dots, \cancel{u_k'},\ldots,u_{m}') \in \bbF[L]^{L}\). Since the same holds for \(f_{k}\), we conclude that \(d_{i+1} \in \bbF[L]^{L}\). Also note that \(d_{i+1} \in \bbF[u_{1}', \dots, \cancel{u_k'},\ldots,u_{m}']\). Finally, observe that
    \[
        d_{i}^{-1} = \dfrac{f_{k}^{s}}{p(u_{1}', \dots, \cancel{u_k'},\ldots,u_{m}')} = \dfrac{f_{k}^{s}}{p(u_{1}', \dots, \cancel{u_k'},\ldots,u_{m}')}\dfrac{f_{k}}{f_{k}} = \dfrac{f_{k}^{s+1}}{d_{i+1}}
    \]
    and that
    \[
        f_{k}^{-1} = \dfrac{p(u_{1}', \dots, \cancel{u_k'},\ldots,u_{m}')}{d_{i+1}}.
    \]
    Therefore,
    \[
        K_{i+1} \cap \bbF[L] \subseteq \bbF[u_{1}', \dots, \cancel{u_k'},\ldots,u_{m}', f_{k}^{-1}, d_{i}^{-1}] \subseteq \bbF[u_{1}', \dots, \cancel{u_k'},\ldots,u_{m}', d_{i+1}^{-1}]
    \]
    which completes the proof.
    %\hfill\(\boxtimes\)
\end{proof}

\section{The Beltrametti--Blasi Formula for nilpotent Lie algebras}
\label{gerAlgInvarRac_consAlg}

The following lemma and corollary together establish the Beltrametti--Blasi formula for nilpotent Lie algebras. Corollary \ref{coro_FLLgrauTras} aligns with the formula presented in \cite{BB66}, which gives 
an upper bound for  the number of functionally independent analytic invariants; here, however, we phrase the result in terms of transcendence degree, rather than the cardinality of the maximal set of functionally independent analytic invariants.

\begin{lemma} \label{lem_annKiNilpLie}
    Suppose that $L$ is a nilpotent Lie algebra with triangular basis 
    \(x_{1}, \dots, x_{n}\). For \(i = 0, 1, \dots, n\), let \(K_{i}\) be as in \eqref{eq:Ki}, with \(K_{0} = \frpolx\) and let \(S_{i}\) be the \(\frpolx\)-submodule of \(\derfrpolx\) generated by \(\ad(x_{1}), \dots, \ad(x_{i})\). Set
    \[
        \Ann(K_{i}) \coloneqq \{D \in \derfrpolx \mid D(f) = 0 \text{ for all } f \in K_{i}\}.
    \]
    Then \(S_{i} = \Ann(K_{i})\) for every \(i = 0, \dots, n\). In particular, for \(i \geq 1\),
    \[
        \frpolx^{\ad(x_{1})} \cap \cdots \cap \frpolx^{\ad(x_{i-1})} \subseteq \frpolx^{\ad(x_{i})} \quad\text{if and only if}\quad \ad(x_{i}) \in S_{i-1}.
    \]
\end{lemma}

\begin{proof}
Let $\bbF(\bm y_n)=\bbF(y_1,\ldots,y_n)$ be the field of rational functions in the variables $y_i$. Suppose that 
$D=\sum_{i=1}^n a_iD_{y_i}\in \textrm{Der}_\bbF(\bbF(\bm y_n))$. Then, for $k\leq n$, 
$D\in\textrm{Ann}(\bbF(y_1,\ldots,y_k))$ if and only if $a_i=D(y_i)=0$ for $i=1,\ldots,k$. That is 
\[ 
\dim \textrm{Ann}(\bbF(y_1,\ldots,y_k))=n-k
\] 
where the dimension is understood as the dimension over the field $\bbF(\bm y_n)$.

Suppose now that $K_i$ is as in the statement of the lemma.  Then, by Theorem \ref{teor_presDerTriangAd}, 
$K_i$ is generated by the triangular system \(z_{1}, \dots, z_{m} \in \polx\) 
with index sequence \(a_{1} < \cdots < a_{m}\). 
Recall that \(z_{j} = x_{a_{j}}q_{j,1} + q_{j,2}\) with \(q_{j,1}, q_{j,2} \in \bbF[x_{1}, \dots, x_{a_{j}-1}]\) 
and \(q_{j,1} \neq 0\). In particular, we have for all $j\in\{1,\ldots,m\}$ that 
$\bbF(x_1,\ldots,x_{a_j})=\bbF(x_1,\ldots,x_{a_j-1},z_j)$. 
Define \(y_{1}, \dots, y_{n}\) by \(y_{a_{j}} = z_{j}\) for \(j = 1, \dots, m\) and \(y_{\ell} = x_{\ell}\) 
for \(\ell \notin \{a_{1}, \dots, a_{m}\}\). Then $y_1,\ldots,y_n$ is an algebraically independent 
system of generators of $\frpolx$ and the argument in the first paragraph of the proof implies that 
\begin{equation}\label{eq:dimann}
\dim \textrm{Ann}(K_i)=\dim \textrm{Ann}(\bbF(y_{a_1},\ldots,y_{a_m}))=n-m.
\end{equation} 

We now prove that \(S_{i} = \Ann(K_{i})\) by induction on \(i\). The claim is clearly true for $i=0$.
 Suppose \(S_{i-1} = \Ann(K_{i-1})\) for some \(i \in \{1, \dots, n\}\). 
 If $\ad(x_i)\in S_{i-1}$, then $K_i=K_{i-1}$ and so 
 \[ 
 \textrm{Ann}(K_i)=\textrm{Ann}(K_{i-1})=S_{i-1}=S_i.
 \] 
 If $\ad(x_i)\not \in S_{i-1}$, then $\dim S_i=\dim S_{i-1}+1$. 
 Furthermore, the transcendence degree of $K_i$ is one less than 
 the transcendence degree of $K_{i-1}$ and so \eqref{eq:dimann} implies that 
 \[ 
 \dim \textrm{Ann}(K_i)=\dim \textrm{Ann}(K_{i-1})+1.
 \] 
 Since 
 $S_i\subseteq \textrm{Ann}(K_i)$ and the dimensions agree,  $S_i=\textrm{Ann}(K_i)$ must hold.
\end{proof}
\begin{corollary} \label{coro_FLLgrauTras}
    Let \(L\) be a nilpotent Lie algebra with basis \(x_{1}, \dots, x_{n}\) and let 
    \[ 
        M_{L} = \big([x_{i}, x_{j}]\big) \in \Mat_{n \times n}\big(\frpolx\big)
    \] 
     be its multiplication matrix. Then \(\bbF(L)^{L}\) is a purely transcendental extension of \(\bbF\) and it holds that 
    \[
        \trdeg\bbF(L)^{L} = \dim(L) - \text{rank}(M_{L}),
    \]
    where \(\trdeg\big(\bbF(L)^{L}\big)\) is the transcendence degree of \(\bbF(L)^{L}\).
\end{corollary}
\begin{proof}
By Theorem \ref{teor_presDerTriangAd}, \(\bbF(L)^{L}\) is a purely transcendental extension of \(\bbF\). In this case, the transcendence degree of \(\bbF(L)^{L}\) is equal to the number of algebraically independent generators given by the theorem. We have that the transcendence degree is equal to 
    \[
        \trdeg\bbF(L)^{L} = \dim(L) - \delta,
    \]
    where \(\delta\) is equal to the number of times the derivation \(\ad(x_{i})|_{K_{i-1}}\) is non-trivial, for each \(i = 2, \dots, n\). We need to show that \(\delta = \text{rank}(M_{L})\). 
    Since the rank of $M_L$ is independent of the chosen basis for $L$, we may assume without 
    loss of generality that $x_1,\ldots,x_n$ is a trianular basis.
    Let \(L_{i}\) be the \(i\)-th row of \(M_{L}\) and note that  $L_i$ is the coefficient vector of 
    $\ad(x_i)$. Also note that \(\ad(x_{i})|_{K_{i-1}} = 0\) means that \(\frpolx^{\ad(x_{1})} \cap \cdots \cap \frpolx^{\ad(x_{i})} = \frpolx^{\ad(x_{1})} \cap \cdots \cap \frpolx^{\ad(x_{i-1})}\), that is, 
    \[ \frpolx^{\ad(x_{1})} \cap \cdots \cap \frpolx^{\ad(x_{i-1})} \subseteq \frpolx^{\ad(x_{i})}.\] Thus, we must show that \(\frpolx^{\ad(x_{1})} \cap \cdots \cap \frpolx^{\ad(x_{i-1})} \subseteq \frpolx^{\ad(x_{i})}\) if and only if \(\ad(x_{i})\) belongs to the submodule of \(\derfrpolx\) generated by \(\ad(x_{1}), \dots, \ad(x_{i-1})\). But this is precisely the equivalence in Lemma \ref{lem_annKiNilpLie}.
    Hence $\delta$ is equal to the number of times $\ad(x_i)$ does not belong to the 
    submodule generated by $\ad(x_1),\ldots,\ad(x_{i-1})$, which is equal to the 
    number of times the row $L_i$ does not belong to the $\frpolx$-subspace generated by $L_1,\ldots,L_{i-1}$. 
    Therefore, \(\trdeg\big(\bbF(L)^{L}\big) = \dim(L) - \text{rank}(M_{L})\).
    %\hfill\(\boxtimes\)
\end{proof}
Corollary \ref{coro_FLLgrauTras} can be obtained from Observations 1.14.13 and 4.9.24 in \cite{Dix96}, for the case where \(L\) is algebraic; there, it is stated (without proof) that the index of \(L\) is equal to the transcendence degree of \(\bbF(L)^{L}\) and also equal to the number \(\dim(L) - \text{rank}(M_{L})\). Moreover, the fact that \(\bbF(L)^{L}\) is a purely transcendental extension is demonstrated in \cite[Proposition 1]{Dix57}. Our proof differs from that of Dixmier by being more constructive.

\section{The algorithm and its implementation to compute invariants of nilpotent Lie algebras}
\label{sec:algorithm}

Given a finite-dimensional nilpotent Lie algebra $L$ with structure constant table 
in a given basis, a triangular basis (as in Section~\ref{gerAlgInvarRac_algoritmoLieNilp}) 
can be obtained by computing a basis $\{x_1,\ldots,x_n\}$ such that 
$\{x_1,\ldots,x_n\}\cap \zeta_i(L)$ is a basis for $\zeta_i(L)$ for all $i$ 
where $\zeta_i(L)$ is the $i$-th term of the upper central series of $L$ starting with 
the center $\zeta_1(L)=Z(L)$ (see, for example, \cite[page~21]{deGraaf}). 

Suppose that $\{x_1,\ldots,x_n\}$  is a triangular basis for a nilpotent Lie algebra $L$ such that 
\[ 
\{x_1,\ldots,x_n\}\cap Z(L)=\{x_1,\ldots,x_{r-1}\}.
\]  
Using the notation of Theorem~\ref{teor_presDerTriangAd}, 
$K_0=\cdots=K_{r-1}=\frpolx$ and the first proper invariant field is $K_{r}$. Now, 
the generators of $K_r$ as in Theorem~\ref{teor_nucPolDerTriIncGerFrac} can be computed directly either by the method of integral curves or 
using the Dixmier map as explained in Sections~\ref{sec:triang} 
and~\ref{derivTriang_invPol}.

However, to compute the subsequent terms $K_{i+1}$ for $i\geq r$, a bit more care is required. 
Recall from Theorem~\ref{teor_presDerTriangAd} that 
$K_i$ is generated by a generating system $z_1,\ldots,z_m$ whose terms are of the form 
\[ 
        z_{j}=x_{a_{j}}q_{j,1}+q_{j,2}
        \] 
such that $1\leq a_1<\cdots<a_m\leq n$, $q_{j,1},q_{j,2}\in \bbF[x_1,\ldots,x_{a_{j}-1}]$ and $q_{j,1}\in \bbF[L]^L\setminus\{0\}$.
Set $D=\ad(x_{i+1})$. We know from Theorem~\ref{teor_presDerTriangAd} that $D$ is a triangular 
derivation, and so $D(z_j)\in\bbF[z_1,\ldots,z_{j-1}]$. In order to 
apply the method of integral curves or the Dixmier map to compute $K_{i+1}=K_i^D$, 
we need to write each $D(z_j)$ as an explicit polynomial in $z_1,\ldots,z_{j-1}$.  
This is possible by the triangular nature of the generating system $z_1,\ldots,z_m$. In fact, 
the system $z_1,\ldots,z_m$ is a SAGBI basis for the algebra $R=\bbF[z_1,\ldots,z_m]$ 
and this fact facilitates membership testing in $R$.
SAGBI bases were introduced in~\cite{RS} and in~\cite{KM}. 
The usefulness of SAGBI bases in computing invariants has been recognised by several 
authors, see for instance~\cite{kuroda,bianchi,stillman,Reichstein}.

The procedure for computing generators for the invariant field of nilpotent Lie algebras 
was implemented by the authors in the computational algebra system SageMath~\cite{Sage}. 
The implementation of the procedure is available in the GitHub repository~\cite{ghrepo}. 
We close this paper by showing a couple of example computations.

\subsection{Nilpotent Lie algebras of low dimension}

Let us first show how to compute the rational invariant field for the standard filiform Lie algebra 
\[ 
L=\left<x,y_0,y_1,y_2,y_3\mid [x,y_i]=y_{i-1}\mbox{ for $i\geq 1$}\right>
\]  
of dimension~5.

\begin{lstlisting}
sage: from examples_lie_algs import standard_filiform_lie_algebra
sage: from rational_invariants import RationalInvariantField
sage: l = standard_filiform_lie_algebra(5)
sage: K = RationalInvariantField(l, has_triangular_basis=True)
sage: [ K.to_symmetric_space(x) for x in K.gens()]
[y0, y0*y2 - 1/2*y1^2, y0^2*y3 - y0*y1*y2 + 1/3*y1^3]\end{lstlisting}

In fact, it is well known  that 
\[
\bbF(x,y_0,y_1,y_2,y_3)^L=\bbF(y_0,-y_1^2 + 2y_0y_2,y_1^3 - 3y_0y_1y_2 + 3y_0^2y_3)
\]
(see, for example, \cite{JS23}).
The algebra $\bbF[x,y_0,y_1,y_2,y_3]^L$ of polynomial invariants has an additional generator: $(z_2^3+z_3^2)/y_0^2$ where $z_2$ and $z_3$ are the second and the third generators in 
the last displayed equation. 

The classification of small-dimensional nilpotent Lie algebras contains parametric families;
for example, in~\cite{CdGSch}, the family $L_{6,19}(\varepsilon)$ is defined as  
    \begin{align*}
    L_{6,19}(\varepsilon)=\left<x_1,\ldots,x_6\mid\right. &[x_5,x_2]=x_1, 
    [x_4,x_3]=x_1,[x_6,x_3]=\varepsilon x_1,\\ 
    &\left.[x_4,x_5]=x_2,[x_4,x_6]=x_3\right>.
    \end{align*}
    Here $\varepsilon\in\bbF$ and if $\varepsilon_1\varepsilon_2\neq 0$, then 
    $L_{6,19}(\varepsilon_1)\cong L_{6,19}(\varepsilon_2)$ if and only if $\varepsilon_1\varepsilon_2^{-1}$ is a square in
 $\bbF$. The rational invariants of $L_{6,19}(\varepsilon)$ can be computed with our 
 implementation as follows.

\begin{lstlisting}[inputencoding=utf8x]
sage: from examples_lie_algs import lie_alg_family_6_19
sage: from rational_invariants import RationalInvariantField
sage: l = lie_alg_family_6_19(QQ)
sage: K = RationalInvariantField(l, has_triangular_basis=True)
sage: [ K.to_symmetric_space(x) for x in K.gens()]
[x1, x1*x6 + (-<@$\varepsilon$@>)*x1*x4 + (-1/2)*x3^2 + (-1/2*<@$\varepsilon$@>)*x2^2]
\end{lstlisting}
Note how the output indicates that the generators of the invariant field depend on  
the parameter $\varepsilon$. 

\subsection{The algebra of strictly upper triangular matrices}
\label{gerAlgInvarRac_matTriagEstr}
In the article \cite{BPP07}, the analytic invariants of \(\fraku_{n}\), the algebra of strictly upper triangular \(n \times n\) matrices, are determined. In it, we have the following result.
\begin{theorem} \label{teor_algTn}
    Let \(E_{ij} \in \Mat_{n\times n}(\bbF)\) be the coordinate function corresponding to the 
    elementary matrix in which the $(i,j)$-entry is one and the other entries are equal to zero.
%    %
%    \[
%        (E_{ij})_{k\ell} \coloneqq \left\{
%        \begin{array}{l}
%            1, \; \text{ if } i = k \text{ and } j = \ell,  \\
%            0, \; \text{ otherwise}.
%        \end{array}
%        \right.
%    \]
    %
    The generators of the algebra of analytic invariants of \(\fraku_{n}\) are
    given by the formal determinants
    \[
        \det\big((E_{ij})_{j=n-k+1, \dots, n}^{i = 1, \dots, k}\big), \; \text{ for } k = 1, \dots, \left\lfloor \dfrac{n}{2} \right\rfloor. \rlap{\hspace{133pt}\(\Box\)}
    \]
\end{theorem}
\begin{example} \label{exem_matTriSupEst8}
    Consider the Lie algebra \(\fraku_{8}\). This algebra has a basis \(\{E_{ij} \mid 1 \le i < j \le 8\}\). Then, by Theorem \ref{teor_algTn}, we have that the algebra of analytic invariants of \(\fraku_{8}\) is generated by the formal determinants of the following matrices.
    \[
        M_{1} =
        \begin{pmatrix}
            E_{18}
        \end{pmatrix},
        \text{\hspace{1em}}
        M_{2} =
        \begin{pmatrix}
            E_{17} & E_{18} \\
            E_{27} & E_{28}
        \end{pmatrix},
        \text{\hspace{1em}}
        M_{3} =
        \begin{pmatrix}
            E_{16} & E_{17} & E_{18} \\
            E_{26} & E_{27} & E_{28} \\
            E_{36} & E_{37} & E_{38}
        \end{pmatrix}
\] 
and  
\[ 
        M_{4} =
        \begin{pmatrix}
            E_{15} & E_{16} & E_{17} & E_{18} \\
            E_{25} & E_{26} & E_{27} & E_{28} \\
            E_{35} & E_{36} & E_{37} & E_{38} \\
            E_{45} & E_{46} & E_{47} & E_{48}
        \end{pmatrix}.
    \]
    For example,
    \[
        \det(M_{1}) = E_{18} \quad \text{ and } \quad \det(M_{2}) = E_{17}E_{28} - E_{18}E_{27}.
    \]
    Arranging the basis \(\{E_{ij} \mid 1 \le i < j \le 8\}\) in a strictly upper triangular matrix \(M\) where the position \((M)_{ij} = E_{ij}\), we have that the matrices \(M_{1}\), \(M_{2}\), \(M_{3}\), and \(M_{4}\) are given as below.
    \begin{center}
        \begin{tikzpicture}
            \draw node at (-0.45,0) {
                \(
                    M = 
                    \begin{pmatrix}
                        0 & E_{12} & E_{13} & E_{14} & E_{15} & E_{16} & E_{17} & E_{18} \\
                        0 & 0 & E_{23} & E_{24} & E_{25} & E_{26} & E_{27} & E_{28} \\
                        0 & 0 & 0 & E_{34} & E_{35} & E_{36} & E_{37} & E_{38} \\
                        0 & 0 & 0 & 0 & E_{45} & E_{46} & E_{47} & E_{48} \\
                        0 & 0 & 0 & 0 & 0 & E_{56} & E_{57} & E_{58} \\
                        0 & 0 & 0 & 0 & 0 & 0 & E_{67} & E_{68} \\
                        0 & 0 & 0 & 0 & 0 & 0 & 0 & E_{78} \\
                        0 & 0 & 0 & 0 & 0 & 0 & 0 & 0
                    \end{pmatrix}
                \)
            };
            \draw[thick, red] (2.6,1.4) rectangle (3.4,2.05);
            \draw[thick, green] (1.75,1) rectangle (3.45,2.15);
            \draw[thick, magenta] (0.8,0.5) rectangle (3.6,2.25);
            \draw[thick, blue] (-0.1,-0.05) rectangle (3.75,2.35);
            \draw node (M1) at (5, 2) {\(M_{1}\)};
            \draw node (M2) at (5, 1) {\(M_{2}\)};
            \draw node (M3) at (5, 0) {\(M_{3}\)};
            \draw node (M4) at (5, -1) {\(M_{4}\)};
            \draw[thick, red, ->] plot [smooth] coordinates {(3.4,1.6) (4,1.7) (4.6,2)};
            \draw[thick, green, ->] plot [smooth] coordinates {(3.45,1) (4,1.5) (4.6,1.2)};
            \draw[thick, magenta, ->] plot [smooth] coordinates {(3.6,0.5) (3.8,0.8) (4.2,0) (4.6,0)};
            \draw[thick, blue, ->] plot [smooth] coordinates {(3.75,-0.05) (4,-0.5) (4.6,-0.9)};
        \end{tikzpicture}
    \end{center}
    The picture illustrates how to read off these invariants.
\end{example}

\begin{lstlisting}[language=python]
sage: from examples_lie_algs 
           import lie_algebra_upper_triangular_matrices
sage: from rational_invariants import RationalInvariantField
sage: l = lie_algebra_upper_triangular_matrices(8, strict=True)
sage: K = RationalInvariantField(l, has_triangular_basis=True)
sage: K.to_symmetric_space(K.gens()[0])
x18
sage: K.to_symmetric_space(K.gens()[1])
x18*x27 - x17*x28
sage: K.to_symmetric_space(K.gens()[2])
x18^2*x27*x36 - x18*x17*x28*x36 - x18^2*x26*x37 + 
x18*x28*x16*x37 + x18*x17*x38*x26 - x18*x16*x27*x38
\end{lstlisting}
Note that there is a fourth generator corresponding to the formal determinant of the matrix
$M_4$, but it is too large to be displayed.

The implementation can compute the invariants of these Lie algebras over $\bbQ$ up to $n=13$ (dimension 78); the computation times for each case are recorded in Table~\ref{tab_tempoCalcTN}. 
All calculations were performed on a computer with a 12th Gen Intel\textsuperscript{\textregistered} Core\texttrademark\ i7-12700H processor, running at 2.4 GHz with 16 GB of memory, using version 10.9 of SageMath.
\begin{table}[H]
    \caption{Time taken to calculate the algebraically independent generators of \(\bbQ(\fraku_{n})^{\fraku_{n}}\).}
    \label{tab_tempoCalcTN}
    \centering
    \renewcommand{\arraystretch}{1.2}
    \begin{tabular}{|C{120pt}|C{45pt}|C{45pt}|C{45pt}|C{45pt}|C{45pt}|}
        \hline
        Matrix size & 3 & 4 & 5 & 6 & 7 \\
        \hline
        Dimension & 3 & 6 & 10 & 15 & 21 \\ 
        \hline 
        Calculation time & 20 ms  & 28 ms & 60 ms & 144 ms & 419 ms \\
        \hline
    \end{tabular}
    \vspace{10pt}

    \renewcommand{\arraystretch}{1.2}
    \begin{tabular}{|C{50pt}|C{50pt}|C{50pt}|C{50pt}|C{50pt}|C{50pt}|}
        \hline
        8 & 9 & 10 & 11 & 12 & 13 \\
        \hline
        28 & 36 & 45 & 55 & 66 & 78\\
        \hline
        856 ms & 2 s & 4.4 s & 9 s & 20 s & 51 s\\
        \hline
        \hline
    \end{tabular}
\end{table}

\end{document}